\documentclass[12pt,a4paper]{amsart}
\usepackage{amsfonts,color}
\usepackage{amsthm}
\usepackage{amsmath}
\usepackage{amscd}
\usepackage[latin2]{inputenc}
\usepackage{t1enc}
\usepackage[mathscr]{eucal}
\usepackage{indentfirst}
\usepackage{graphicx}
\usepackage{graphics}
\usepackage{pict2e}
\usepackage{epic}
\usepackage{url}
\usepackage{epstopdf}

\numberwithin{equation}{section}
\usepackage[margin=2.6cm]{geometry}

 \usepackage{tikz}
\usetikzlibrary{shapes}
\tikzstyle{vertex}=[draw=black,circle,fill=black,minimum size=4pt, inner sep=0pt, outer sep=0pt,text=white,line width=0mm]
\tikzstyle{c0}=[shape=circle, minimum size=4pt, fill=white]
\tikzstyle{c1}=[shape=rectangle, minimum size=7pt, fill=red]
\tikzstyle{c2}=[shape=diamond, minimum size=10pt, fill=blue]

\theoremstyle{plain}
\newtheorem{Th}{Theorem}[section]
\newtheorem{Lemma}[Th]{Lemma}
\newtheorem{Cor}[Th]{Corollary}

 \theoremstyle{definition}
\newtheorem{Def}[Th]{Definition}

\newtheorem{Rem}[Th]{Remark}
\newtheorem{?}[Th]{Problem}

\newcommand{\h}{\textbf{h}}

\begin{document}

\title{Extremal regular graphs: the case of the infinite regular tree}

\author[P. Csikv\'ari]{P\'{e}ter Csikv\'{a}ri}

\address{MTA-ELTE Geometric and Algebraic Combinatorics Research Group \\ P\'{a}zm\'{a}ny P\'{e}ter s\'{e}t\'{a}ny 1/C \\ Hungary \& E\"{o}tv\"{o}s Lor\'{a}nd University \\ Mathematics Institute, Department of Computer 
Science \\ H-1117 Budapest
\\ P\'{a}zm\'{a}ny P\'{e}ter s\'{e}t\'{a}ny 1/C \\ Hungary \& Alfr\'ed R\'enyi Institute of Mathematics \\ H-1053 Budapest \\ Re\'altanoda utca 13-15} 

\email{peter.csikvari@gmail.com}

\thanks{The author  is partially supported by the Hungarian National Research, Development and Innovation Office, NKFIH grant K109684 and NN114614,  a Slovenian-Hungarian grant, by the MTA R\'enyi "Lend\"ulet" Groups and Graphs Research Group, and by the ERC Consolidator Grant  648017.}



\begin{abstract} In this paper we study the following problem. Let $A$ be a fixed graph, and let $\hom(G,A)$ denote the number of homomorphisms from a graph $G$ to $A$. Furthermore, let $v(G)$ denote the number of vertices of $G$, and let $\mathcal{G}_d$ denote the family of $d$--regular graphs. The general problem studied in this paper is to determine 
$$\inf_{G\in \mathcal{G}_d}\hom(G,A)^{1/v(G)}.$$
It turns out that in many instances the infimum is not achieved by a finite graph, but a sequence of graphs with girth (i. e., length of the shortest cycle) tending to infinity. In other words, the optimization problem is solved by the infinite $d$--regular tree.

We prove this type of  results for the number of independent sets of bipartite graphs, evaluations of the Tutte-polynomial, Widom-Rowlinson configurations, and many more graph parameters. Our main tool will be a transformation called $2$-lift.
\end{abstract}

\maketitle

\section{Introduction} 

Let $P(G)$ be a graph parameter specified later which has size roughly $c^{v(G)}$, where $v(G)$ denotes the number of vertices of a graph $G$. For instance, the number of spanning trees, number of (perfect) matchings, 
number of independent sets or the number of homomorphisms into a fixed graph $A$. It is a general problem in extremal graph theory to study
\begin{align}
\sup P(G)^{1/v(G)}\ \ \ \ \mbox{and}\ \ \ \ \inf P(G)^{1/v(G)}
\end{align}
where the supremum and infimum are taken among all $d$--regular (bipartite) graphs. Let $\mathcal{G}_d$ denote the family of $d$--regular graphs, and similarly let $\mathcal{G}^b_d$ denote the family of $d$--regular bipartite graphs. 

It turns out that the answer often (but far from always) involves one of the following three graphs: the complete graph $K_{d+1}$, the complete bipartite graph $K_{d,d}$, and the infinite $d$--regular tree $\mathbb{T}_d$. Here the meaning of the first two cases is clear, and subsequently we will explain what it means that the infinite $d$--regular tree $\mathbb{T}_d$ is an extremal graph. 

Below we give some examples for all cases. J. Kahn \cite{Kahn} showed that if one considers the number of independent sets $I(G)$, then
\begin{align}
\sup_{G\in \mathcal{G}^b_d}I(G)^{1/v(G)}=I(K_{d,d})^{1/v(K_{d,d})}.
\end{align}
In other words, for any $d$--regular bipartite graph $G$ we have 
\begin{align}
I(G)^{1/v(G)}\leq I(K_{d,d})^{1/v(K_{d,d})}.
\end{align}
It turns out one can drop the condition of bipartiteness in J. Kahn's theorem. Y. Zhao \cite{Zhao1} used a clever trick to reduce the general case to the bipartite case. He compared $G$ with $G\times K_2$ which is defined as follows: its vertex set is $V(G)\times \{0,1\}$, and for $u,v\in V(G)$ the vertices $(u,i),(v,j)\in V(G)\times \{0,1\}$ form an edge of $G\times K_2$ if and only if $(u,v)\in E(G)$ and $i+j=1$. Note that if $G$ is a $d$--regular graph then $G\times K_2$ is also $d$--regular, and in addition, it is bipartite too. Later we will introduce the concept of $2$-lift and we will see that $G\times K_2$ and $G\cup G$ are both $2$-lifts of $G$.

\begin{Th}[Y. Zhao \cite{Zhao1}] For any graph $G$, we have
\begin{align}
I(G\times K_2)\geq I(G)^2.
\end{align}
Consequently,  we have
\begin{align}
I(G)^{1/v(G)}\leq I(G\times K_2)^{1/v(G\times K_2)}\leq I(K_{d,d})^{1/v(K_{d,d})},
\end{align}
where the second inequality follows from J. Kahn's result.
\end{Th}

It turns out that the number of independent sets is a special instance of a larger class of graph parameters, namely the number of homomorphisms into a fixed graph $A$. Recall that if $G$ and $A$ are graphs then a map $\varphi: V(G)\to V(A)$ is a homomorphism if $(\varphi(u),\varphi(v))\in E(A)$ whenever $(u,v)\in E(G)$.  Let $\hom(G,A)$ denote the number of homomorphisms from the graph $G$ to the graph $A$. Note that if $A_{\mathrm{ind}}$ is an edge with a loop at one of its end vertices then $\hom(G,A_{\mathrm{ind}})=I(G)$, the number of independent sets of $G$.
Indeed, the vertices which maps to the vertex of $A_{\mathrm{ind}}$ without loop have to form an independent set in $G$. Note that if $A=K_q$, then $\hom(G,A)$ counts the number of proper colorings of $G$ with $q$ colors.

When $A_{\mathrm{WR}}=P_3^{\circ}$, a path on $3$ vertices with a loop added at each vertex then $\hom(G,A_{\mathrm{WR}})$ counts the number of Widom-Rowlinson configurations. 

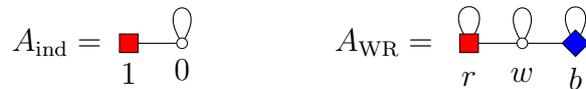
\begin{figure}[h]
\centering
\[ A_{\mathrm{ind}} = 
\begin{tikzpicture}[scale=.7, baseline=-1mm]
  \node[vertex, c1, label={below:$1$}] (v1) at (-1,0) {};
  \node[vertex, c0, label={below:$0$}] (v0) at (0,0) {};
  \draw (v1) to (v0) .. controls +(60:1) and +(120:1) .. (v0);
\end{tikzpicture}
\qquad
\qquad
 A_{\mathrm{WR}} = 
\begin{tikzpicture}[scale=.7, baseline=-1mm]
  \node[vertex, c1, label={below:$\vphantom{b}r$}] (v1) at (-1,0) {};
  \node[vertex, c0, label={below:$\vphantom{b}w$}] (v0) at (0,0) {};

  \node[vertex, c2, label={below:$b$}] (v2) at (1,0) {};
  \draw (v1) .. controls +(60:1) and +(120:1) .. (v1) to
        (v0) .. controls +(60:1) and +(120:1) .. (v0) to
        (v2) .. controls +(60:1) and +(120:1) .. (v2);
\end{tikzpicture}
\]
\caption{The target graphs for the hard-core model (independent sets) and the 
Widom-Rowlinson model.}
\label{fig:H-examples}
\end{figure}

E. Cohen, W. Perkins and P. Tetali \cite{CPT} (for simpler proofs see also \cite{CCPT,S}) proved that in this case $K_{d+1}$ is the maximizing graph:
\begin{align}
\sup_{G\in \mathcal{G}_d}\hom(G,A_{\mathrm{WR}})^{1/v(G)}=\hom(K_{d+1},A_{\mathrm{WR}})^{1/v(K_{d+1})}.
\end{align}
It turns out that $K_{d+1}$ is sometimes the minimizing graph. For instance, J. Cutler and J. Radcliffe \cite{CR} proved that
\begin{align}
\inf_{G\in \mathcal{G}_d}I(G)^{1/v(G)}=I(K_{d+1})^{1/v(K_{d+1})}.
\end{align}
It is also known (for details see \cite{Zhao3}) that
\begin{align}
\inf_{G\in \mathcal{G}_d}\hom(G,K_q)^{1/v(G)}=\hom(K_{d+1},K_q)^{1/v(K_{d+1})}.
\end{align}
For many more examples see the recent survey of Y. Zhao \cite{Zhao3} and the references therein.
\bigskip

On the other hand, this paper is not about the extremality of $K_{d+1}$ and $K_{d,d}$. This paper is about the extremality of the infinite $d$--regular tree. To enlighten this phenomenon we give two theorems which together gives an interesting theorem about $q$-colorings of bipartite graphs. 

\begin{Th}[P. Csikv\'ari and Z. Lin \cite{CL}] \label{count_coloring} For any bipartite graph $G$ with $v(G)$ vertices and $e(G)$ edges we have 
\begin{align}
\hom(G,K_q)\geq q^{v(G)}\left(\frac{q-1}{q}\right)^{e(G)}.
\end{align}
In particular, if $G$ is a $d$--regular bipartite graph then
\begin{align}
\hom(G,K_q)^{1/v(G)}\geq q\left(\frac{q-1}{q}\right)^{d/2}.
\end{align}
Furthermore, if the graph $G$ contains $\varepsilon n$ vertex disjoint cycles of length at most $\ell$, then there is some $c_q(\varepsilon, \ell)>1$ such that
\begin{align}
\hom(G,K_q)\geq c_q(\varepsilon, \ell)^{v(G)}q^{v(G)}\left(\frac{q-1}{q}\right)^{e(G)}.
\end{align}
\end{Th}

\begin{Th}[A. Bandyopadhyay and D. Gamarnik \cite{BG}]
Let $q\geq d+1$. Let $(G_i)$ be a sequence of $d$--regular graphs such that the girth $g(G_i)$ (the length of the shortest cycle) tends to infinity. Then
\begin{align}
\lim_{i\to \infty}\hom(G_i,K_q)^{1/v(G_i)}=q\left(\frac{q-1}{q}\right)^{d/2}.
\end{align}
\end{Th}

The two theorems together show that for $q\geq d+1$ we have
\begin{align}
\inf_{G\in \mathcal{G}_d^b}\hom(G,K_q)^{1/v(G)}=q\left(\frac{q-1}{q}\right)^{d/2},
\end{align}
but the infimum is not achieved by a finite graph.
\bigskip

Throughout the whole paper let $g(G)$ denote  the girth of a graph $G$, i. e., the length of the shortest cycle of the graph $G$.
\medskip

It turns out that the case when $(G_i)$ is a sequence of $d$--regular graphs such that $g(G_i)\to \infty$ is a particular case of a Benjamini--Schramm convergent graph sequence, and there is a limit object which is the infinite $d$--regular tree $\mathbb{T}_d$ in this case. Benjamini--Schramm convergence will be explained in Section~\ref{BS-conv}. The above example motivates the following definition.

\begin{Def} For a graph parameter $P(.)$ let
\begin{align}
"P(\mathbb{T}_d)^{1/v(\mathbb{T}_d)}"=\inf_{(G_i)}\liminf_{g(G_i)\to \infty \atop G_i\in \mathcal{G}_d}P(G_i)^{1/v(G_i)},
\end{align}
and let
\begin{align}
"P(\mathbb{T}^b_d)^{1/v(\mathbb{T}^b_d)}"=\inf_{(G_i)}\liminf_{g(G_i)\to \infty \atop G_i\in \mathcal{G}^b_d}P(G_i)^{1/v(G_i)}.
\end{align}
(The letter b stands for bipartite.)
\end{Def}

In the above definitions it is somewhat inconvenient that we need to take an extra infimum even after liminf. One would like to see simply a limit. Unfortunately, we know that in general there is no limit in such an expression. On the other hand, the author hopes that whenever we can prove an extremal result then we actually have a limit in these definitions. 

Note that one can rewrite the definition of $"P(\mathbb{T}_d)^{1/v(\mathbb{T}_d)}"$ as follows:
\begin{align}
"P(\mathbb{T}_d)^{1/v(\mathbb{T}_d)}"=\lim_{k\to \infty}\inf_{g(G)\geq k}P(G)^{1/v(G)}.
\end{align}
This definition looks simpler, but it hides the real problem, namely that we would like to determine this quantity. Note that the statement
\begin{align}
P(G)^{1/v(G)}\geq  "P(\mathbb{T}_d)^{1/v(\mathbb{T}_d)}"
\end{align}
is really a combinatorial statement, it simply means that for any $k$ there is some graph $G_k$ such that the girth of $G_k$ is at least $k$ and
\begin{align}
P(G)^{1/v(G)}\geq P(G_k)^{1/v(G_k)}.
\end{align}
\bigskip

Now we are ready to give some theorems which will play  exemplary roles. 

\begin{Th} \label{Th-A} Let $A_{\mathrm{WR}}=P_3^o$, the path on $3$ vertices with a loop added at each vertices.  Let $G$ be a $d$--regular graph. Then
\begin{align}
\hom(G,A_{\mathrm{WR}})^{1/v(G)}\geq "\hom(\mathbb{T}_d,A_{\mathrm{WR}})^{1/v(\mathbb{T}_d)}".
\end{align}
\end{Th}

\noindent Later we will extend Theorem~\ref{Th-A} to a graph class $\mathcal{A}$. The next theorem is also an example for this phenomenon, but it is not a new theorem. In a bit different form it appears in \cite{Ruo}.

\begin{Th}\label{Th-B}
Let $I(G)$ denote the number of independent sets of a graph $G$. Then for any $d$--regular bipartite graph $G$  
we have
\begin{align}
I(G)^{1/v(G)}\geq "I(\mathbb{T}^b_d)^{1/v(\mathbb{T}^b_d)}".
\end{align}
\end{Th}

\noindent Recall that $I(G)=\hom(G,A_{\mathrm{ind}})$. Later we will extend Theorem~\ref{Th-B} to a graph class $\mathcal{B}$. It turns out that for the number of independent sets the value $"I(\mathbb{T}^b_d)^{1/v(\mathbb{T}^b_d)}"$ can be determined exactly and this is a limit for every bipartite graph sequence of $d$--regular graphs $(G_i)$ with $g(G_i)\to \infty$. For details and an explicit version of Theorem~\ref{Th-B} see Subsection~\ref{case:ind}.
\bigskip

Let $Z(G,q,w)$ be the following statistical physical version of the Tutte-polynomial (also called dichromatic polynomial):
\begin{align}
Z(G,q,w)=\sum_{F\subseteq E(G)}q^{k(F)}w^{e(F)},
\end{align}
where $k(F)$ is the number of components of the subgraph $(V(G),F)$, and $e(F)=|F|$ is the number of edges.

The connection with the usual form of the Tutte-polynomial is the following:
\begin{align}
T(G,x,y)=(x-1)^{-k(E)}(y-1)^{-v(G)}Z(G,(x-1)(y-1),y-1).
\end{align}

We will prove the following theorem about the statistical physical version of the Tutte-polynomial. This theorem also appeared in \cite{Ruo2} in a slightly different form.

\begin{Th} \label{Tutte}
Let $Z(G,q,w)$ be defined as above, and assume that $q\geq 1$ and $w\geq 0$. Then 
\begin{align}
Z(G,q,w)^{1/v(G)}\geq "Z(\mathbb{T}_d,q,w)^{1/v(\mathbb{T}_d)}".
\end{align}
\end{Th}

\noindent In Subsection~\ref{case:Tutte} we will give a more explicit form of this theorem.
\bigskip

\subsection{Methods.} The key tool in this paper is the transformation $2$-lift. The definition of a $2$-lift is given below.

\begin{Def} Let $G$ be a graph. Then $H$ is a $2$-lift of $G$, if $V(H)=V(G)\times \{0,1\}$, and for every $(u,v)\in E(G)$, exactly one of the following two pairs are edges of $H$: $((u,0),(v,0))$ and $((u,1),(v,1))\in E(H)$, or $((u,0),(v,1))$ and $((u,1),(v,0))\in E(H)$. If $(u,v)\notin E(G)$, then none of $((u,0),(v,0))$, $((u,1),(v,1))$, $((u,0),(v,1))$ and $((u,1),(v,0))$ are edges in $H$.

More generally one can define a $k$-lift (or $k$-cover) of a graph as follows. The vertex set of a $k$-lift $H$ is $V(H)=V(G)\times \{0,1,\dots, k-1\}$, and if $(u,v)\in E(G)$ then we choose a perfect matching between the vertices $(u,i)$ and $(v,j)$ for $0\leq i,j\leq k-1$. If $(u,v)\notin E(G)$, then there is no edges between $(u,i)$ and $(v,j)$ for $0\leq i,j\leq k-1$.
\end{Def}

There are two notable special cases of a $2$-lift. When all edges are of the form $((u,0),(v,0))$ and $((u,1),(v,1))$ then we simply get two disjoint copies of the graph $G$, so it is just $G\cup G$. When all edges are of the form  $((u,0),(v,1))$ for some $(u,v)\in E(G)$ then we get $G\times K_2$. It will turn out that these  special $2$-lifts often play the role of an extremal graph among $2$-lifts of a graph.
\bigskip

Recall that for graphs $G$ and $A$, a map $\varphi:V(G)\to V(A)$ is called a homomorphism if $(\varphi(u),\varphi(v))\in E(A)$ whenever $(u,v)\in E(G)$. The number of homomorphisms from $G$ to $A$ is denoted by $\hom(G,A)$. One can identify the graph $A$ with its adjacency matrix and we get that $\hom(G,A)$ is a special case of the following definition valid for any symmetric matrix $A$, not just $0-1$ matrices.

\begin{Def} Let $A=(a_{ij})$ be a $q\times q$ symmetric matrix. Then
\begin{align}
Z(G,A)=\sum_{\varphi:V(G)\to [q]}\prod_{(u,v)\in E(G)}a_{\varphi(u),\varphi(v)},
\end{align}
where the notation $[q]$ stands for $\{1,2,\dots ,q\}$.
\end{Def}

So if we identify $A$ with its adjacency matrix we get that $Z(G,A)=\hom(G,A)$. On the other hand, the expression $Z(G,A)$ often shows up in statistical physics. Below we list some notable matrices for which we get the so-called partition function of some well-known statistical physical model.

$$A_{\mathrm{Is(\beta)}}=\left( \begin{array}{cc} e^{\beta} & e^{-\beta} \\ e^{-\beta} & e^{\beta} \end{array}\right), \,  \,
A_{\mathrm{ind}}=\left( \begin{array}{cc} 1 & 1 \\ 1 & 0 \end{array}\right),\, \,
A_{\mathrm{WR}}=\left( \begin{array}{ccc} 1 & 1 & 0\\ 1 & 1 & 1 \\ 0 & 1 & 1 \end{array}\right),$$
$$A_{q}(w)=\left( \begin{array}{cccc} 1+w & 1 & 1 & 1 \\ 1 & 1+w & 1 & 1 \\ 1 & 1 & 1+w & 1 \\ 1 & 1 & 1 & 1+w \end{array}\right).
$$
The $A_q(w)$ is a matrix of size $q\times q$, in the picture $A_4(w)$ is depicted. 
\medskip

The expression $Z(G, A_{\mathrm{Is(\beta)}})$ is the partition function of the Ising-model. When $\beta>0$ we speak about ferromagnetic regime, and when  $\beta<0$ then we speak about antiferromagnetic regime. It is well-known in statistical physics that the model behaves very differently in the two regimes. It will turn out that even in our paper this matters: we will prove that when $\beta>0$ then
\begin{align}
Z(G\cup G,A_{\mathrm{Is(\beta)}})\geq Z(H,A_{\mathrm{Is(\beta)}})
\end{align}
for any $2$-lift $H$ of $G$. This result also covered by a result of Ruozzi \cite{Ruo}. For $\beta<0$ we will prove that 
\begin{align}
Z(G\times K_2,A_{\mathrm{Is(\beta)}})\geq Z(H,A_{\mathrm{Is(\beta)}})
\end{align}
for any $2$-lift $H$ of $G$.
\medskip

The expression $Z(G, A_{\mathrm{ind}})$ counts the number of independent sets of $G$. As we remarked, Y. Zhao \cite{Zhao1} showed that for any graph $G$ we have
\begin{align}
Z(G\times K_2,A_{\mathrm{ind}})\geq Z(G\cup G,A_{\mathrm{ind}}).
\end{align}
We will show that in fact for any $2$-lift $H$ of $G$ we have
\begin{align}
Z(G\times K_2,A_{\mathrm{ind}})\geq Z(H,A_{\mathrm{ind}}).
\end{align}
In statistical mechanics counting independent sets corresponds to the hard-core model, and $Z(G,A_{\mathrm{ind}})$ is the partition function of the hard-core model. In general, statistical physicists also introduce a weight function: let $i_k(G)$ the number of independent sets of size $k$, then let
\begin{align}
I(G,\lambda)=\sum_{k=0}^{v(G)}i_k(G)\lambda^k.
\end{align}
With a slight extension of our definition of $Z(G,A)$ we can cover this case too: let us introduce a weight function $\nu:[q]\to \mathbb{R}_+$ and let
\begin{align}
Z(G,A,\nu)=\sum_{\varphi:V(G)\to [q]}\prod_{u\in V(G)}\nu(\varphi(u))\cdot \prod_{(u,v)\in E(G)}a_{\varphi(u),\varphi(v)}.
\end{align}
In the hard-core model $\nu(1)=1$ and $\nu(2)=\lambda$. One might wonder why we did not introduce immediately $Z(G,A,\nu)$. There is a good reason: absolutely unnecessary, in most of our theorems the following is true: if the theorem applies for some matrix $A$, then it is immediately true for the weighted version! See Section~\ref{weights} for the details.
\medskip

The expression $Z(G, A_{\mathrm{WR}})$ is the partition function of the Widom-Rowlinson model. This is the number of colorings with red (color $1$), white 
(color $2$) and blue (color $3$) such that red and blue vertices cannot be adjacent. In this case, we will prove that
\begin{align}
Z(G\cup G,A_{\mathrm{WR}})\geq Z(H,A_{\mathrm{WR}})
\end{align}
for any $2$-lift $H$ of $G$.
\medskip

The above discussion motivates the following definition.

\begin{Def}
Let $\mathcal{A}$ be the family of matrices $A$ for which
\begin{equation} \label{ineq-gen}
Z(G\cup G,A)\geq Z(H,A)
\end{equation}
for every graph $G$ and its $2$-lift $H$. Let $\mathcal{A}^b$ be the family of matrices $A$ for which
inequality~\ref{ineq-gen} holds for every bipartite graph $G$ and its $2$-lift $H$.

Finally, let $\mathcal{B}$ be the family of matrices $A$ for which
\begin{equation} \label{ineq-rev}
Z(G\times K_2,A)\geq Z(H,A)
\end{equation}
for every graph $G$ and its $2$-lift $H$. 

With some slight abuse of notation we say that a graph $A$ belongs to  $\mathcal{A}$ (resp. $\mathcal{A}^b$ or $\mathcal{B}$) if its adjacency matrix belongs to $\mathcal{A}$ (resp. $\mathcal{A}^b$ or $\mathcal{B}$).
\end{Def}

So we see two different behaviors so far: for the ferromagnetic Ising-model and the Widom-Rowlinson model the graph $G\cup G$ is the maximizing $2$-lift, in other words, $A_{\mathrm{Is(\beta)}},A_{\mathrm{WR}}\in \mathcal{A}$ for $\beta>0$. While for the antiferromagnetic Ising-model and the hard-core model the graph $G\times K_2$ is the maximizing graph among the $2$-lifts, in other words, $A_{\mathrm{Is(\beta)}},A_{\mathrm{ind}}\in \mathcal{B}$ for $\beta<0$.
In order to understand the difference between these models the following definition will be useful.

\begin{Def} A matrix $A$ is called $TP_2$-matrix if every $2\times 2$ submatrix (not just principal) of $A$ has a non-negative determinant. In other words, if $i<j$ and $r<s$ then the matrix $\left(\begin{array}{cc} a_{ir} & a_{is} \\ a_{jr} & a_{js} \\ \end{array} \right)$ has a non-negative determinant, i. e., $a_{ir}a_{js}-a_{is}a_{jr}\geq 0$. 

A matrix $A$ is called $TN_2$-matrix if every $2\times 2$ submatrix of $A$ has a non-positive determinant.
\end{Def}

We remark that the properties $TP_2$ and $TN_2$ heavily depend on the ordering of the rows and columns of the matrix. This is slightly inconvenient. In Theorem~\ref{gen} we will extend these notions in such a way that the obtained concept will not be sensitive to the permutations of rows and columns. 

Note that the Ising-model has a $TP_2$-matrix for $\beta\geq 0$, and $TN_2$-matrix for $\beta\leq 0$, and in general every $2\times 2$ matrix is either $TP_2$ or $TN_2$. The Widom-Rowlinson model has a $TP_2$ matrix. Y. Zhao proved that if $A$ is a loop-threshold graph then
\begin{align}
Z(G\times K_2,A)\geq Z(G\cup G,A).
\end{align}
A graph is loop-threshold if there exists a real number $w_i$ for each vertex $i$, and an $\alpha$ such that $(i,j)\in E(G)$ if and only if $w_i+w_j\leq \alpha$.
Loop-threshold graphs have an adjacency matrix where all $1$'s are condensing to the top-left corner. 

$$A_{\mathrm{thr}}=\left( \begin{array}{cccccc} 
1 & 1 & 1 & 1 & 1 & 1\\ 
1 & 1 & 1 & 1 & 1 & 0\\ 
1 & 1 & 1 & 1 & 0 & 0\\ 
1 & 1 & 1 & 1 & 0 & 0\\ 
1 & 1 & 0 & 0 & 0 & 0\\ 
1 & 0 & 0 & 0 & 0 & 0\\ 
\end{array}\right).
$$
\centerline{An example for an adjacency matrix of loop-threshold graph.}
\medskip

Extending Y. Zhao's theorem We will show that
\begin{align}
Z(G\times K_2,A)\geq Z(H,A)
\end{align}
for any loop-threshold graph $A$ and any $2$-lift $H$ of $G$, in other words, loop-threshold graphs are in $\mathcal{B}$.

It is easy to see that loop-threshold graphs have $TN_2$ adjacency matrices. Surprisingly, there is a similar class for $TP_2$-matrices: here the $1$'s are at around the diagonal or in other words, the $0$'s are condensing to the bottom-left and the top-right corners. Let us call these graphs thick paths. (If we put a loop to each vertex of a path we get such a graph.) Formally, we can define these graphs very similarly to the loop-threshold graphs: a graph is a thick-path if  there exists a real number $w_i$ for each vertex $i$, and an $\alpha$ such that $(i,j)\in E(G)$ if and only if $|w_i-w_j|\leq \alpha$.

$$A_{\mathrm{th. paths}}=\left( \begin{array}{cccccc} 
1 & 1 & 1 & 0 & 0 & 0\\ 
1 & 1 & 1 & 1 & 0 & 0\\ 
1 & 1 & 1 & 1 & 1 & 0\\ 
0 & 1 & 1 & 1 & 1 & 0\\ 
0 & 0 & 1 & 1 & 1 & 1\\ 
0 & 0 & 0 & 0 & 1 & 1\\ 
\end{array}\right).
$$
\centerline{An example for an adjacency matrix of thick-path graph.}
\medskip

For thick-path graphs $A$ we will show that
\begin{align}
Z(G\cup G,A)\geq Z(H,A)
\end{align}
for any $2$-lift $H$ of $G$, in other words, thick-path graphs are in $\mathcal{A}$.
\bigskip

After this introduction the following theorems are not surprising.

\begin{Th} \label{pos}
Let $A=(a_{ij})$ be a $q\times q$ non-negative symmetric $TP_2$-matrix. Let $G$ be a graph, and let $H$ be any $2$-lift of $G$. Then
\begin{align}
Z(G,A)^2=Z(G\cup G,A)\geq Z(H,A).
\end{align}
In other words, every non-negative symmetric $TP_2$-matrix is in $\mathcal{A}$.
In particular,
\begin{align}
Z(G,A)^2\geq Z(G\times K_2,A).
\end{align}
\end{Th}

\begin{Th} \label{neg}
Let $A=(a_{ij})$ be a $q\times q$ non-negative symmetric $TN_2$-matrix. Let $G$ be a graph, and let $H$ be any $2$-lift of $G$. Then
\begin{align}
Z(G\times K_2,A)\geq Z(H,A).
\end{align}
In other words, every symmetric non-negative symmetric $TN_2$-matrix is in $\mathcal{B}$.
In particular,
\begin{align}
Z(G,A)^2\leq Z(G\times K_2,A).
\end{align}
\end{Th}

Note that for a bipartite graph $G$ the graphs $G\cup G$ and $G\times K_2$ are isomorphic. This implies that $\mathcal{B}\subseteq \mathcal{A}^b$. In particular, we get the following corollary.

\begin{Cor} \label{neg2}
Let $A=(a_{ij})$ be a $q\times q$ non-negative symmetric $TN_2$-matrix. Let $G$ be a bipartite graph, and let $H$ be any $2$-lift of $G$. Then
\begin{align}
Z(G,A)^2=Z(G\cup G,A)\geq Z(H,A).
\end{align}
\end{Cor}

So in case of a bipartite graph $G$ the $TP_2$ and $TN_2$-matrices produce the same inequality
\begin{align}
Z(G\cup G,A)\geq Z(H,A),
\end{align}
in spite of the fact that the general inequality (i. e. for a non-bipartite graph $G$) between $Z(G\cup G,A)$ and  $Z(G\times K_2,A)$ is reversed.
\medskip

We remark that Theorems~\ref{pos} and \ref{neg} are not the strongest theorems one can say. In fact, the main theorem of this paper is Theorem~\ref{gen} which we will discuss in Section~\ref{general setting}. The reason why we do not discuss this theorem here is that it is quite technical, one needs some preparation even to phrase it. 
\bigskip

Finally, for a positive integer $q$, and the matrix $A_q(w)$ we have
\begin{align}
Z(G,q,w)=Z(G,A_q(w)),
\end{align}
the statistical physical version of the Tutte-polynomial. This is also the partition function of the Potts-model. Unfortunately, it is neither $TP_2$-matrix, nor $TN_2$.
In spite of this, we will show that even without assuming the integrality of $q$ we have the following result. This result also appeared in Ruozzi's work \cite{Ruo2} in a more general form. He proved the same result for any $k$-lift. We will give a brief account of the work of Ruozzi in Section~\ref{Tutte-polynomial}.

\begin{Th} \label{Tutte-2}
Let $G$ be a graph, and let $H$ be any $2$-lift of $G$. Then for any $q\geq 1$ and $w\geq 0$ we have
\begin{align}
Z(G,q,w)^2=Z(G\cup G,q,w)\geq Z(H,q,w).
\end{align}
\end{Th}

\subsection{Back to extremal graph theory.} \label{egt} We have seen that  J. Kahn proved that for any $d$--regular bipartite graph $G$ on $n$ vertices one has 
\begin{align}
Z(G,A_{\mathrm{ind}})\leq Z(K_{d,d},A_{\mathrm{ind}})^{v(G)/2d},
\end{align}
and he conjectured that one can drop the condition of bipartiteness. This turned out to be indeed true: Y. Zhao showed that
\begin{align}
Z(G,A_{\mathrm{ind}})^2\leq Z(G\times K_2,A_{\mathrm{ind}}),
\end{align}
since $G\times K_2$ is bipartite, combined with J. Kahn's result this immediately gave the desired result. On the other hand, Galvin and Tetali \cite{GT} extended Kahn's result by showing that for any graph $A$ and $d$--regular bipartite graph $G$ on $n$ vertices one has 
\begin{align}
Z(G,A)\leq Z(K_{d,d},A)^{v(G)/2d}.
\end{align}
An alternative proof of this fact can be found in \cite{LZ}, this proof also works for a non-negative matrix $A$.

This prompted Y. Zhao \cite{Zhao2} to study that for which graphs $A$ one can say that
\begin{align}
Z(G,A)^2\leq Z(G\times K_2,A).
\end{align}
For all these graphs (or matrices) $A$, the function $Z(G,A)^{1/v(G)}$ is maximized at $K_{d,d}$ among $d$--regular graphs. These graphs were further studied by Sernau \cite{S}. In this sense, the extension for arbitrary $2$-lift and the counterpart for $G\cup G$ instead of $G\times K_2$ seems to be useless. Surprisingly this is not the case, because another method in extremal graph theory was developed for giving lower bounds on the quantity $Z(G,A)^{1/v(G)}$. The idea very briefly is the following: assume that for any $2$-lift $H$ of $G$ we have
\begin{align}
Z(G,A)^2=Z(G\cup G,A)\geq Z(H,A).
\end{align}
A key observation made by Linial \cite{nati} is that for any graph $G$ one can  construct a sequence of graphs $G=G_0,G_1,G_2,\dots $ such that $G_{i+1}$ is a $2$-lift of $G_i$, and the girth $g(G_i)$ tend to infinity. This observation can be combined with the above inequality as in \cite{csi1}.
This way we get two things
\begin{align}
Z(G,A)^{1/v(G)}\geq Z(G_1,A)^{1/v(G_1)} \geq Z(G_2,A)^{1/v(G_2)}\geq \dots 
\end{align}
and the sequence $g(G_i)$ tends to infinity. When the graph $G$ is $d$--regular then all $(G_i)$ are $d$--regular too, and if $G$ is bipartite then so all $G_i$. (In other words, the constructed sequence $(G_i)$ is Benjamini--Schramm convergent to the infinite $d$--regular tree in case of a $d$--regular graph $G$. In general, $(G_i)$ converges to a distribution on the rooted universal cover trees of $G$.) In particular,
\begin{align}
Z(G,A)^{1/v(G)} \geq "Z(\mathbb{T}_d,A)^{1/v(\mathbb{T}_d)}".
\end{align}
Let us summarize it as a theorem:

\begin{Th}\label{general}
(a) Let $P(G)$ be a fixed graph parameter. If for any graph $G$ and its $2$-lift $H$ we have
\begin{align}
P(G)^2\geq P(H),
\end{align}
then for any $d$--regular graph $G$ we have 
\begin{align}
P(G)^{1/v(G)}\geq "P(\mathbb{T}_d)^{1/v(\mathbb{T}_d)}".
\end{align}
(b) If for any bipartite graph $G$ and its $2$-lift $H$ we have
\begin{align}
P(G)^2\geq P(H),
\end{align}
then for any $d$--regular bipartite graph $G$ we have 
\begin{align}
P(G)^{1/v(G)}\geq "P(\mathbb{T}^b_d)^{1/v(\mathbb{T}^b_d)}".
\end{align}
\end{Th}

\bigskip

If it were true that
\begin{align}
\lim_{i\to \infty} Z(G_i,A)^{1/v(G_i)}
\end{align}
exists and we can compute it then we would get a general lower bound for $Z(G,A)^{1/v(G)}$. 
Interestingly, this is a well-studied problem, especially in locally tree-like graphs, exactly the case we need. For instance, it is known that the limit exists and it is computed for $A_{\mathrm{Is(\beta)}}$ if $\beta>0$, or $\beta<0$ and $G_i$'s are bipartite, or for $A_{\mathrm{ind}}$ again when $G_i$'s are bipartite. For these models these are exactly the cases when we were able to prove an inequality of type $Z(G,A)^2=Z(G\cup G,A)\geq Z(H,A)$. In Section~\ref{limit} we will return to this problem, where we gather a few known results.
\bigskip

This paper is organized as follows. In Section~\ref{warm-up} we study the number of independent sets and matchings of $2$-lifts, in particular we prove Theorem~\ref{Th-B}. In Section~\ref{posneg} we prove Theorem~\ref{pos} and Theorem~\ref{neg}. In Section~\ref{general setting} we give an extension of these theorems. In Section~\ref{weights} we extend our results to vertex-weighted partition functions. In Section~\ref{Tutte-polynomial}  we prove Theorem~\ref{Tutte-2} and we also give an account to Ruozzi's work. In Section~\ref{other work} we elaborate how our work is related to some known ideas, most notably to the work of Y. Zhao and L. Sernau. In Section~\ref{limit} we summarize some known results on the limit values of the sequence $(Z(G_i,A)^{1/v(G_i)})_{i=1}^{\infty}$ and combine it with our results. In Section~\ref{the end} we finish the paper with some remarks and open problems.

\section{Warm-up: independent sets and matchings} \label{warm-up}

In this section we consider the case $A=A_{\mathrm{ind}}$, i. e., we are counting independent sets. This section is completely elementary.

\begin{Th} \label{independent} Let $G$ be a graph, and let $H$ be a $2$-lift of $G$. Then
\begin{align}
i_k(H)\leq i_k(G\times K_2),
\end{align}
where $i_k(.)$ denotes the number of independent sets of size $k$.
\end{Th}

\begin{Rem} This statement gives a generalization of Yufei Zhao's result, namely 
\begin{align}
i_k(G\cup G)\leq i_k(G\times K_2).
\end{align}
On the other hand, if $G$ is bipartite then $G\times K_2=G\cup G$ in which case it gives that
\begin{align}
i_k(G \cup G)\geq i_k(H)
\end{align}
for any $2$-lift $H$. 
\end{Rem}

\begin{proof}
Let $I$ be any independent set of a $2$-lift of $G$. Let us consider the projection of $I$ to $G$, then it will consist of vertices and "double-vertices" (i.e, when two vertices map to the same vertex). Let $\mathcal{R}$ be the set of these configurations. Then 
\begin{align}
i_k(H)=\sum_{R \in \mathcal{R}}|\phi_H^{-1}(R)|
\end{align}
and
\begin{align}
i_k(G\times K_2)=\sum_{R \in \mathcal{R}}|\phi_{G\times  K_2}^{-1}(R)|,
\end{align}
where $\phi_H$ and $\phi_{G\times K_2}$ are the projections from $H$ and $G\times K_2$ to $G$. Note that
\begin{align}
|\phi_{G\times K_2}^{-1}(R)|=2^{k(R)},
\end{align}
where $k(R)$ is the number of connected components of $R$ different from a double-vertex. Indeed, in each component we can lift the vertices such a way that the image belongs to exactly one bipartite class. The projection of a  double-vertex must be a connected component on its own. On the other hand,
\begin{align}
|\phi_{H}^{-1}(R)|\leq 2^{k(R)},
\end{align}
since in each component if we know the inverse image of one vertex then we immediately know the inverse images of all other vertices. Clearly, there is no equality in general. 
Hence 
\begin{align}
|\phi_{H}^{-1}(R)|\leq |\phi_{G\times K_2}^{-1}(R)|\end{align}
and consequently,
\begin{align}
i_k(H)\leq i_k(G\times K_2).
\end{align}
\end{proof}

\begin{proof}[Proof of Theorem~\ref{Th-B}] This immediately follows from Theorem~\ref{independent} and Theorem~\ref{general}.
\end{proof}

The following theorem is not in the scope of this paper, but we mention it for two reasons. Its proof is practically the same as the above proof as we will see. This theorem generalizes the well-known fact that
\begin{align}
m_k(G\times K_2)\geq m_k(G\cup G),
\end{align}
and in case of bipartite graphs we have
\begin{align}
m_k(G \cup G)\geq m_k(H)
\end{align}
for any $2$-lift $H$ which was proved in \cite{csi1}.

\begin{Th} Let $G$ be a graph, and let $H$ be a $2$-lift of $G$. Then
\begin{align}
m_k(H)\leq m_k(G\times K_2),
\end{align}
where $m_k(.)$ denotes the number of matchings of size $k$.
\end{Th}

\begin{proof}
Let $M$ be any matching of a $2$-lift of $G$. Let us consider the projection of $M$ to $G$, then it will consist of paths, cycles and "double-edges" (i.e, when two edges project to the same edge). Let $\mathcal{R}$ be the set of these configurations. Then 
\begin{align}
m_k(H)=\sum_{R \in \mathcal{R}}|\phi_H^{-1}(R)|
\end{align}
and
\begin{align}
m_k(G\times K_2)=\sum_{R \in \mathcal{R}}|\phi_{G\times  K_2}^{-1}(R)|,
\end{align}
where $\phi_H$ and $\phi_{G\times K_2}$ are the projections from $H$ and $G\times K_2$ to $G$. Note that
\begin{align}
|\phi_{G\times K_2}^{-1}(R)|=2^{k(R)},
\end{align}
where $k(R)$ is the number of paths and cycles of $R$. Indeed, in each path or cycle we can lift the edges in two different ways. The projection of a  double-edge is naturally unique. On the other hand,
\begin{align}
|\phi_{H}^{-1}(R)|\leq 2^{k(R)},
\end{align}
since in each path or cycle if we know the inverse image of one edge then we immediately know the inverse images of all other edges. Clearly, there is no equality for cycles in general. 
Hence 
\begin{align}
|\phi_{H}^{-1}(R)|\leq |\phi_{G\times K_2}^{-1}(R)|
\end{align}
and consequently,
\begin{align}
m_k(H)\leq m_k(G\times K_2).
\end{align}
\end{proof}

\section{Proofs of Theorems~\ref{pos} and ~\ref{neg}.} \label{posneg}

In this section we prove Theorems~\ref{pos} and ~\ref{neg}. We will need the following lemma.

\begin{Lemma} \label{pos-lemma} For every edge $e$ of a graph $G$ let $A(e)=\left( \begin{array}{cc} a_{11}(e) & a_{12}(e) \\ a_{21}(e) & a_{22}(e) \end{array}\right)$ be a non-negative matrix.
Let $f_e(x): \{-1,1\}\to \mathbb{R}$ defined as follows:
\begin{align}
f_e(x)=\left\{\begin{array}{cc} a_{11}(e)a_{22}(e) & \mbox{if}\ x=1 \\
                                a_{12}(e)a_{21}(e) & \mbox{if}\ x=-1 \\
               \end{array} \right..
\end{align}
							
Let $G$ be a graph. Assume that for all edge $e=(u,v)$ we have $\det(A(e))\geq 0$. 
Then for any $\underline{s}=(s_{u,v})_{(u,v)\in E(G)}\in \{-1,1\}^{E(G)}$ we have
\begin{align}
\sum_{\underline{\sigma} \in \{-1,1\}^{V(G)}}\prod_{(u,v)\in E(G)}f_e(s_{u,v}\sigma_u\sigma_v)\leq \sum_{\underline{\sigma} \in \{-1,1\}^{V(G)}} \prod_{(u,v)\in E(G)}f_e(\sigma_u\sigma_u),
\end{align}
where $\underline{\sigma}=(\sigma_u)_{u\in V(G)}$.
On the other hand, if for all edge $e=(u,v)\in E(G)$ we assume that $\det(A(e))\leq 0$ then
\begin{align}
\sum_{\underline{\sigma} \in \{-1,1\}^{V(G)}}\prod_{(u,v)\in E(G)}f_e(s_{u,v}\sigma_u\sigma_v)\leq \sum_{\underline{\sigma} \in \{-1,1\}^{V(G)}} \prod_{(u,v)\in E(G)}f_e(-\sigma_u\sigma_v).
\end{align}

\end{Lemma}

\begin{proof} Note that
\begin{align}
f_e(x)=\frac{a_{11}(e)a_{22}(e)+a_{12}(e)a_{21}(e)}{2}+x\frac{a_{11}(e)a_{22}(e)-a_{12}(e)a_{21}(e)}{2}.
\end{align}
For sake of simplicity let us call 
\begin{align}
c_1(e)=\frac{a_{11}(e)a_{22}(e)+a_{12}(e)a_{21}(e)}{2}\ \ \mbox{and}\ \ c_2(e)=\frac{a_{11}(e)a_{22}(e)-a_{12}(e)a_{21
}(e)}{2}.
\end{align}
Then
\begin{align}
\sum_{\underline{\sigma} \in \{-1,1\}^{V(G)}}\prod_{(u,v)\in E(G)}f_e(s_{u,v}\sigma_u\sigma_v) &=
 \sum_{\underline{\sigma} \in \{-1,1\}^{V(G)}}\prod_{(u,v)\in E(G)}(c_1(e)+s_{u,v}\sigma_u\sigma_vc_2(e))=\\
&=\sum_{\underline{\sigma} \in \{-1,1\}^{V(G)}}\sum_{F\subseteq E(G)}\left(\prod_{(u,v)\notin F}c_1(e)\right)
\left(\prod_{(u,v)\in F}(s_{u,v}\sigma_u\sigma_vc_2(e))\right)=\\
&=\sum_{F\subseteq E(G)}\left(\prod_{(u,v)\notin F}c_1(e)\right)\left(\prod_{(u,v)\in F}s_{u,v}c_2(e)\right)\left(\sum_
{\underline{\sigma} \in \{-1,1\}^{V(G)}}\prod_{(u,v) \in F}\sigma_u\sigma_v\right).
\end{align}
Let $G_F$ be the graph with vertex set $V(G)$ and edge set $F$. Let $O(G_F)$ be the set of vertices which have odd degree in the graph. Then
\begin{align}
\sum_{\underline{\sigma} \in \{-1,1\}^{V(G)}}\prod_{(u,v)\in F}\sigma_u\sigma_v=
\sum_{\underline{\sigma} \in \{-1,1\}^{V(G)}}\prod_{u\in O(G_F)}\sigma_u=\left\{ \begin{array}{ll} 2^{|V(G)|} & \mbox{if}\ O(G_F)=\emptyset ,\\
0 & \mbox{otherwise}. \\
\end{array} \right. 
\end{align}
Hence
\begin{align}
\sum_{\underline{\sigma} \in \{-1,1\}^{V(G)}}\prod_{(u,v)\in E(G)}f_e(s_{u,v}\sigma_u\sigma_v)=2^{|V(G)|} \sum_{F\subseteq E(G) \atop O(G_F)=\emptyset}\prod_{(u,v)\notin F}c_1(e)\prod_{(u,v)\in F}s_{u,v}c_2(e).
\end{align}
If $c_2(e)\geq 0$ for all $e\in E(G)$ then it is clearly maximized when $(s_{u,v})_{(u,v)\in E(G)}=\underline{1}$. If $c_2(e)\leq 0$ for all $e\in E(G)$, then $\prod_{(i,j)\in F}s_{u,v}c_2(e)$ is positive if all $s_{u,v}=-1$, so in this case the function is maximized at $(s_{u,v})_{(u,v)\in E(G)}=-\underline{1}$.
\end{proof}

\begin{Rem} If $c_2(e)\leq 0$ for all $e\in E(G)$, but $G$ is bipartite then observe that the graph $G_F$ has even number of edges as it is a bipartite Eulerian graph. Hence $\prod_{(u,v)\in F}c_2(e)\geq 0$ for all such graphs, hence the function is again maximized when $(s_{u,v})_{(u,v)\in E(G)}=\underline{1}$, and at the same time at $-\underline{1}$.
\end{Rem}

\begin{proof}[Proof of Theorem~\ref{pos} and Theorem~\ref{neg}.] Before we start proving the theorem it is worth introducing a few notations. If $H$ is a fixed $2$-lift of $G$ then let
\begin{align}
s_{u,v}=\left\{ \begin{array}{ll} 1 & \mbox{if}\ ((u,0),(v,0))\ \mbox{and} \ ((u,1),(v,1))\in E(H), \\
-1 & \mbox{if}\ ((u,0),(v,1))\ \mbox{and} \ ((u,1),(v,0))\in E(H). \\
\end{array} \right. 
\end{align}
For a $\varphi: V(H)\to [q]$ let 
\begin{align}
S_0=\{u\in V(G)\ | \varphi((u,0))=\varphi((u,1))\}
\end{align}
and 
\begin{align}
S_1=\{u\in V(G)\ | \varphi((u,0))\neq \varphi((u,1))\}.
\end{align}
For a $\varphi: V(H)\to [q]$ and an $u\in S_1$ let
\begin{align}
\sigma_{\varphi}(u)=\left\{ \begin{array}{ll} 1 & \mbox{if}\ \varphi((u,0))<\varphi((u,1)), \\
-1 & \mbox{if}\ \varphi((u,0))>\varphi((u,1)) \\
\end{array} \right. 
\end{align}
Finally for $\varphi: V(H)\to [q]$ let $[\varphi]$ denote the equivalence class of maps $\varphi$ for which the set system $\{\varphi((u,0)),\varphi((u,1))\}$ for all $u\in V(G)$ is the same.
Furthermore, for $e=(u,v)$ and $\varphi: V(H)\to [q]$ let $t_1=\min(\varphi((u,0)),\varphi((u,1))),t_2=\max(\varphi((u,0)),\varphi((u,1)))$,
$s_1=\min(\varphi((v,0)),\varphi((v,1))$, and $s_2=\max(\varphi((v,0)),\varphi((v,1)))$
\begin{align}
f_{e,\varphi}(x)=\left\{\begin{array}{cc} a_{t_1,s_1}a_{t_2,s_2} & \mbox{if}\ x=1 \\
                                a_{t_1,s_2}a_{t_2,s_1} & \mbox{if}\ x=-1 \\
               \end{array} \right..
\end{align}

In other words, $f_{e,\varphi}(x)$ is the function belonging to the matrix $A_{e,\varphi}=\left( \begin{array}{cc} a_{t_1,s_1} & a_{t_1,s_2} \\ a_{t_2,s_1} & a_{t_2,s_2}\end{array}\right)$ in the lemma. Clearly, $f_{e,\varphi}(x)$ depends only on $[\varphi]$ so we will write $f_{e,[\varphi]}(x)$ instead of it.

With these notations we have
\begin{align}
Z(H,A) &=\sum_{\varphi:V(G)\to [q]}\prod_{(u',v')\in E(H)}a_{\varphi(u'),\varphi(v')}=\\
&=\sum_{[\varphi]}\sum_{\varphi\in [\varphi]}\prod_{(u',v')\in E(H)}a_{\varphi(u'),\varphi(v')}=\\
&=\sum_{[\varphi]}\sum_{\varphi\in [\varphi]}\prod_{(u',v')\in E(H) \atop \{u',v'\}\cap S_0\neq \emptyset}a_{\varphi(u'),\varphi(v')}\prod_{(u',v')\in E(H) \atop u',v'\in S_1}a_{\varphi(u'),\varphi(v')}=\\
&=\sum_{[\varphi]}\sum_{\varphi\in [\varphi]}\prod_{(u',v')\in E(H) \atop \{u',v'\}\cap S_0\neq \emptyset}a_{\varphi(u'),\varphi(v')}\prod_{e=(u,v)\in E(G)}f_{e,[\varphi]}(s_{u,v}\sigma_{\varphi}(u)\sigma_{\varphi}(v)).
\end{align}
Note that the term
\begin{align}
\prod_{(u',v')\in E(H) \atop \{u',v'\}\cap S_0\neq \emptyset}a_{\varphi(u'),\varphi(v')}
\end{align}
only depends on $[\varphi]$, but it does not depend on the $2$-lift $H$ we consider. So we can denote it by $w([\varphi])$. Then
\begin{align}
Z(H,A)=\sum_{[\varphi]}w([\varphi])\sum_{\varphi\in [\varphi]}\prod_{e=(u,v)\in E(S_1)}f_{e,[\varphi]}(s_{u,v}\sigma_{\varphi}(u)\sigma_{\varphi}(v)).
\end{align}

Clearly, we have $w([\varphi])\geq 0$. By the lemma we know that
\begin{align}
\sum_{\varphi\in [\varphi]}f_{e,[\varphi]}(s_{u,v}\sigma_{\varphi}(u)\sigma_{\varphi}(v))=\sum_{(\sigma_{\varphi}(u))_u\in \{\pm 1\}^{S_1}}\prod_{(i,j)\in E(S_1)}f_{e,[\varphi]}(s_{u,v}\sigma_{\varphi}(u)\sigma_{\varphi}(v))
\end{align}
is maximized at $(s_{u,v})=\underline{1}$ if $\det(A_{e,\varphi})\geq 0$ for all $e$. This means that $Z(H,A)$ is maximized when $H=G\cup G$. On the other hand, if $\det(A_{e,\varphi})\leq 0$ for all $e$ then the above function is maximized at $(s_{u,v})=-\underline{1}$ which means that 
$Z(H,A)$ is maximized when $H=G\times K_2$.
\end{proof}

\section{More general setting} \label{general setting}

In this section we will frequently use the following definition.

\begin{Def} A matrix-decorated graph is a graph $G$ together with a symmetric matrix $A_e$ of size $q\times q$ assigned to every edge $e$. 
We will denote a decorated matrix by $(G|A_e)$.

The homomorphism function of a decorated graph $(G|A_e)$ is defined as
\begin{align}
\h(G|A_e)=\sum_{\varphi: V(G)\to [q]}\prod_{e\in E(G)}A_e(\varphi(u),\varphi(v)).
\end{align}
\end{Def}

The point of this definition is that instead of considering a $2$-lift $H$ of a graph $G$, and its homomorphisms into a matrix $A$, we will consider the decoration of $G$ with two matrices introduced later, $A^=$ and $A^{\times}$, such that $\h(G|A_e)=Z(H,A)$.

Indeed, let $A^=$ be the following matrix of size $q^2\times q^2$: its rows and columns are denoted by the ordered pairs $(i,j)$, where $i,j\in [q]$, and
\begin{align}
A^{=}((i,j),(k,l))=A(i,k)\cdot A(j,l).
\end{align}
In other words, $A^{=}$ is simply the tensor product $A\otimes A$. Let $A^{\times}$ be the following matrix of size $q^2\times q^2$: its rows and columns are again denoted by the ordered pairs $(i,j)$, where $i,j\in [q]$, and
\begin{align}
A^{\times}((i,j),(k,l))=A(i,l)\cdot A(j,k).
\end{align}
So $A^{\times}$ is the skew tensor product of $A$ with itself.

Now if $H$ is a $2$-lift of $G$, then write the matrix $A^{=}$ to those edges of $G$, where the edges of $H$ are $((u,0),(v,0))$ and $((u,1),(v,1))$, and write $A^{\times}$ to those edges of $G$, where the edges of $H$ are $((u,1),(v,0))$ and $((u,0),(v,1))$. Then if we consider a map $\varphi: H\to A$, then we can introduce $\tilde{\varphi}:V(G)\to [q^2]$ such that $\tilde{\varphi}(u)=(\varphi((u,0)),\varphi((u,1))$. Then
\begin{align}
Z(H,A)=\sum_{\varphi:V(H)\to [q]}\prod_{(u',v')\in E(H)}A(\varphi(u'),\varphi(v'))=\sum_{\tilde{\varphi}:V(G)\to [q^2]}\prod_{(u,v)\in E(G)}A_e(\varphi(u),\varphi(v))=\h(G|A_e),
\end{align}
where $A_e$ is $A^=$ or $A^{\times}$ according to the above rule.

Next let us introduce the matrices $E$ and $D$:
\begin{align}
E=\frac{1}{2}\left(A^=+A^{\times}\right)\, \, \, \mbox{and}\, \, \, D=\frac{1}{2}\left(A^=-A^{\times}\right).
\end{align}
Let us write $E$ and $D$ as block matrices with the convention that the first $q$ rows and columns correspond to the elements $(i,i)$, where $i\in [q]$, the next $\binom{q}{2}$ rows and columns correspond to the elements $(i,j)$, where $i<j$, $i,j\in [q]$, and the last $\binom{q}{2}$ rows and columns correspond to the elements $(j,i)$, where $i<j$, $i,j\in [q]$. Then
\begin{align}
E=\left( \begin{array}{ccc} E_{0} & E_{01} & E_{01} \\ E_{01}^T & E_1 & E_1 \\ E_{01}^T & E_1 & E_1 \end{array}\right) \, \, \, \mbox{and} \, \, \, D=\left( \begin{array}{ccc} 0 & 0 & 0 \\ 0 & D_1 & -D_1 \\ 0 & -D_1 & D_1 \end{array}\right).
\end{align}

Having these notations we are able to phrase the main theorem of this paper.

\begin{Th} \label{gen} Let $A$ be a non-negative symmetric matrix of size $q\times q$, and let the matrices $A^=,A^{\times},D,E$ be defined as above. 
\medskip

\noindent (a) If there exists a diagonal matrix $S$ of size $q^2\times q^2$ with entries $\pm 1$ in the diagonal such that 
$SDS$ has only non-negative entries then for any $2$-lift $H$ of $G$ we have
\begin{align}
Z(G\cup G,A)\geq  Z(H,A).
\end{align}
\medskip

\noindent (b) If there exists a diagonal matrix $S$ of size $q^2\times q^2$ with entries $\pm 1$ in the diagonal such that 
$SDS$ has only non-positive entries then for any $2$-lift $H$ of $G$ we have
\begin{align}
Z(G\times K_2,A)\geq  Z(H,A).
\end{align}

\end{Th}

\begin{Rem} It is easy to check that the condition of the existence of a diagonal matrix $S$ of size $q^2\times q^2$ with entries $\pm 1$ in the diagonal such that 
$SDS$ has only non-negative (non-positive) entries is equivalent with the existence of a diagonal matrix $S_1$ of size $\binom{q}{2}\times \binom{q}{2}$ with entries $\pm 1$ in the diagonal such that $S_1D_1S_1$ has only non-negative (non-positive) entries. One direction is trivial: the restriction of $S$ to the corresponding rows and columns implies the existence of $S_1$. The other direction follows from the following argument. Let $c(i,j)=1$ if $i<j$, and $c(i,j)=-1$ if $i>j$. 
Let us define the diagonal matrix $S$ as follows:
\begin{align}
s_{(i,j),(i,j)}=(s_1)_{(i,j),(i,j)}c(i,j),
\end{align}
and $s_{(i,i),(i,i)}=1$ (it does not matter how we define it). 
Then it is easy to check that $SDS$ is non-negative if $S_1D_1S$ is non-negative, and 
$SDS$ is non-positive if $S_1D_1S$ is non-positive.

\end{Rem}

Theorem~\ref{gen} might seem to be a very technical statement, but it covers both Theorem~\ref{pos} and Theorem~\ref{neg} by simply choosing $S$ to be the identity matrix. The relationship of part (b) with Y. Zhao's theorem will be explained later. 
\bigskip

Let us collect a few lemmas. The first one is trivial, but very useful.

\begin{Lemma} The function $\h(G|A_e)$ is linear in each matrix, i. e., if for some edge $e_1$ we have $A_{e_1}=\beta B_{e_1}+\gamma C_{e_1}$ then 
\begin{align}
\h(G|A_e (e\in E(G)\setminus {e_1}),A_{e_1})=\beta\h(G|A_e (e\in E(G)\setminus {e_1}),B_{e_1})+\gamma \h(G|A_e (e\in E(G)\setminus {e_1}),C_{e_1}).
\end{align}

\end{Lemma}

\begin{Lemma} Assume that $\h(G|A_e)\neq 0$ and for every edge $e$ we have $A_e=D$ or $E$. Then the subgraph of $G$ with vertex set $V(G)$ and edge set $E_D(G)=\{e\in E(G)\ |\ A_e=D\}$ is an Eulerian subgraph, i. e., all degrees are even.

\end{Lemma}

\begin{proof} Suppose for contradiction that for some vertex $x$ the degree of $x$ in the graph $G_D=(V(G),E_D(G))$ is odd. Then for those $\tilde{\varphi}$ maps for which $\tilde{\varphi}(v)=(i,i)$  for some $i$, the contribution of the product
\begin{align}
\prod_{(u,v)\in E(G)}A_e(\tilde{\varphi}(u),\tilde{\varphi}(v))=0.
\end{align}
Moreover, if we change some $\tilde{\varphi}(v)=(i,j)$ to $\tilde{\varphi}(v)=(j,i)$, the contribution changes to $(-1)^k=-1$ times the original, where $k=\deg_{G_D}(x)$. Hence $\h(G|A_e)=0$, contradiction.
\end{proof}

Note that
\begin{align}
A^==\frac{1}{2}\left(E+D\right)\, \, \, \mbox{and}\, \, \, A^{\times}=\frac{1}{2}\left(E-D\right).
\end{align}
Using these equations and the linearity of the function $\h(G|A_e)=Z(H,A)$ with $A_e=A^=$ or $A^{\times}$, we get that  $\h(G|A_e)$  is a sum of similar expression such that each $A_e$ is $D$ or $E$. By the above lemma the contribution of those $\h$'s for which $G_D$ is not Eulerian is $0$.
The next lemma will immediately imply Theorem~\ref{gen}.

\begin{Lemma} \label{genlemma} (a) Assume that there exists  some diagonal matrix $S$ of size $q^2\times q^2$ with entries $\pm 1$ in the diagonal such that $SDS$ has only non-negative entries. Then $\h(G|A_e)\geq 0$ if all $A_e=D$ or $E$.
\medskip

\noindent (b) Assume that there exists  some diagonal matrix $S$ of size $q^2\times q^2$ with entries $\pm 1$ in the diagonal such that $SDS$ has only non-positive entries. Then $\h(G|A_e)\geq 0$ if all $A_e=-D$ or $E$.

\end{Lemma}

\begin{proof} (a) For an $x\in \mathbb{R}$ let $x^+=\max(x,0)$, and $x^-=-\min(x,0)$. Then $x=x^+-x^-$. For the matrix $D=(d_{ij})$ let $K=(d_{ij}^+)$ and $L=(d_{ij}^-)$.
Then $D=K-L$. Note that $S$ determines a bipartation of the $q^2$ pairs, namely let $(i,j)\in A$ if $s_{(i,j),(i,j)}=1$,  and let $(i,j)\in B$ if $s_{(i,j),(i,j)}=-1$. Since $SDS$ is a non-negative matrix then we can arrange the rows and columns of $D$ such a way that first few columns and rows correspond to the elements of $A$, and the last columns and rows correspond to the elements of $B$. Then we can write $K$ and $L$ into the form
\begin{align}
K=\left( \begin{array}{cc} K_{11} & 0 \\ 0 & K_{-1,-1} \end{array} \right)\, \, \, \mbox{and} \, \, \, L=\left( \begin{array}{cc} 0 & L_{1,-1}  \\  L_{1,-1}^T & 0 \end{array}\right).
\end{align}
Now we are ready to start the proof. First of all, by the previous lemma  we can assume that $G_D$ is an Eulerian-graph. Since $D=K-L$ we can use the linearity of the function $\h(G|A_e)$ to decompose each $D$ to $2$ terms $K$ and $-L$. So we get $2^{|E(G_D)|}$ decorated graphs where each edge is decorated with $E$, $K$ or $-L$. We need two observations.

One crucial observation is the following: if the edges which are decorated by $-L$ does not form a cut of $G_D$ then $\h(G|A_e)=0$. Indeed, if for some $\tilde{\varphi}$ the contribution is not $0$ then two elements of $A$ (or $B$) should be connected by an edge with matrix $E$ or $K$, and two element from different classes should be connected by an edge with matrix $E$ or $-L$. This means that the edges equipped with the matrix $-L$ form a cut of the graph $G_D$.

The second observation is the following: a cut of an Eulerian graph (namely $G_D$) has to contain an even number of edges, this means that we can delete the minus signs in front of $-L$. Now we see that $\h(G|A_e)$ has to be non-negative since $E,K,L$ are non-negative matrices.

\noindent (b) The proof of part (b) is practically the same  as of part (a). 

\end{proof}

Now the proof of Theorem~\ref{gen} immediately follows.

\begin{proof}[Proof of Theorem~\ref{gen}.] We prove only part (a), part (b) is completely analogous. Note that
\begin{align}
Z(G\cup G,A)=\h(G|A_e),
\end{align}
where all $A_e=A^==\frac{1}{2}\left(E+D\right)$. While
\begin{align}
Z(H,A)=\h(G|A_e),
\end{align}
where some $A_e=A^==\frac{1}{2}\left(E+D\right)$, and for others we have $A_e=A^{\times}=\frac{1}{2}\left(E-D\right)$.
So if we expand both function into $2^{e(G)}$ terms then $Z(G\cup G,A)$ will contain only $\h(G|A_e=D\, \mbox{or}\, E)$ with positive 
coefficients, while $Z(H,A)$ may contain the same terms with negative coefficients. By part (a) of Lemma~\ref{genlemma} the terms $\h(G|A_e=D\, \mbox{or}\, E)$ are non-negative hence
\begin{align}
Z(G\cup G,A) \geq Z(H,A).
\end{align}
\end{proof}

\section{Vertex weighted partition functions} \label{weights}

In this section we prove the following result.

\begin{Th} \label{gen-weight} Let $A$ be a non-negative symmetric matrix of size $q\times q$, and let $\nu:[q]\to \mathbb{R}_+$ be a weight function.
Furthermore, let the matrices $A^=,A^{\times},D,E$ be defined as above. 
\medskip

\noindent (a) If there exists a diagonal matrix $S$ of size $q^2\times q^2$ with entries $\pm 1$ in the diagonal such that 
$SDS$ has only non-negative entries then for any $2$-lift $H$ of $G$ we have
\begin{align}
Z(G\cup G,A,\nu)\geq  Z(H,A,\nu).
\end{align}
\medskip

\noindent (b) If there exists a diagonal matrix $S$ of size $q^2\times q^2$ with entries $\pm 1$ in the diagonal such that 
$SDS$ has only non-positive entries then for any $2$-lift $H$ of $G$ we have
\begin{align}
Z(G\times K_2,A,\nu)\geq  Z(H,A,\nu).
\end{align}

\end{Th}

\begin{proof} First let us assume that $\nu:[q]\to \mathbb{Z}_+$. Let us define the following block matrix: replace the element $a_{ij}$ with a block of size $\nu(i)\times \nu(j)$ whose each element is $a_{ij}$. Let $A^{\nu}$ be the obtained matrix, this is a matrix of size $Q=\sum_{i=1}^q\nu(i)$. It is easy to see that $Z(G,A,\nu)=Z(G,A^{\nu})$ for every graph $G$. If we create the matrices $(A^{\nu})^=,(A^{\nu})^{\times},D^{\nu},E^{\nu},D_1^{\nu}$ from $A^{\nu}$ then we see that $D^{\nu}$ again satisfies the condition of part (a) or part (b) if $D$ satisfies it. Simply in $S$ we need to change a $\pm 1$ at place $(i,j)$ to $\nu(i)\nu(j)$ pieces of $\pm 1$'s. Hence in part (a) we have
\begin{align}
Z(G\cup G,A,\nu)=Z(G\cup G,A^{\nu})\geq Z(H,A^{\nu})=Z(H,A,\nu).
\end{align}
And in part (b) we have
\begin{align}
Z(G\times K_2,A,\nu)=Z(G\times K_2,A^{\nu})\geq  Z(H,A^{\nu})=Z(H,A,\nu).
\end{align}
Next let as assume that $\nu:[q]\to \mathbb{Q}_+$. Then let us choose an $R$ for which $R\nu(i)\in \mathbb{Z}_+$ for all $i\in [q]$. Then
\begin{align}
Z(G,A,\nu)=\frac{1}{R^{v(G)}}Z(G,A,R\nu)
\end{align}
for every graph $G$ and we are done since we know that for $R\nu$ the statement is true. Finally, for any $\nu:[q]\to \mathbb{R}_+$ we get the statement by continuity.
\end{proof}

\section{Tutte-polynomial} \label{Tutte-polynomial}

In this section we prove Theorem~\ref{Tutte-2} which directly implies Theorem~\ref{Tutte}. The theorem is based on the FKG-inequality for the random cluster model. In the random cluster model we have a fixed graph $G$ and we choose a random subset $F$ of the edge set with probability proportional to $q^{k(F)}w^{|F|}$, i. e., we have
\begin{equation}
\mathbb{P}(F)=\frac{q^{k(F)}w^{|F|}}{Z(G,q,w)}.
\end{equation}
Clearly, the probability of the event that a fixed edge $e$ is not in the chosen set $F$ is
\begin{equation}
\mathbb{P}(e\notin F)=\frac{Z(G-e,q,w)}{Z(G,q,w)},
\end{equation}
where $G-e$ is the graph obtained by deleting the edge $e$ from $G$.
On the other hand, the probability that $e$ is in the random set $F$ is 
\begin{equation}
\mathbb{P}(e\in F)=\frac{wZ(G/e,q,w)}{Z(G,q,w)},
\end{equation}
where $G/e$ is the graph which we get if we contract the edge $e$. (Note that it is worth working with multigraphs, i. e., we allow multiple edges and loops too.)
\medskip

Note that there is a natural partial ordering on the subsets of $F$, namely $F'<F$ if $F'\subseteq F$. We say that a function $f$ on the subset of the edges is  monotone increasing if $f(F')\leq f(F)$ whenever $F'<F$, and monotone decreasing if $f(F')\geq f(F)$ whenever $F'<F$. It turns out that when $q\geq 1$ and $w\geq 0$ then the random cluster model satisfies the FKG lattice condition and consequently it implies that
\begin{equation}
\mathbb{E}(f)\cdot \mathbb{E}(g)\leq \mathbb{E}(fg)
\end{equation}
for all monotone increasing functions. For details see Theorem 4.11 and Theorem 8.7 of \cite{Grim}.

We only need the special case when $f=1_e$ and $g=1_f$, the indicator functions of the events that $e$ or $f$ is in the random subset $F$, these are clearly  monotone increasing functions. In this case we get that
\begin{equation}
\mathbb{P}(e\in F)\mathbb{P}(f\in F)\leq \mathbb{P}(e,f\in F),
\end{equation}
which implies that
\begin{equation}
\mathbb{P}(e\in F)\mathbb{P}(f\notin F)\geq \mathbb{P}(e\in F, f\notin F),
\end{equation}
and also the inequality
\begin{equation}
\mathbb{P}(e,f\in F)\mathbb{P}(e,f\notin F)\geq \mathbb{P}(e\in F, f\notin F)\mathbb{P}(e\notin F, f\in F).
\end{equation}
If we apply this last inequality for some graph $H$ and we write it back to the function $Z(H,q,w)$ we get that
\begin{equation} \label{pos-cor}
Z(H-\{e,f\},q,w)Z(H/\{e,f\},q,w)\geq Z((H-e)/f,q,w)Z((H/e)-f,q,w).
\end{equation}
Clearly, if $H=G\cup G$ and $e$ and $f$ are in different copies of $G$ then we have equality in the above inequality.

\begin{proof}[Proof of Theorem~\ref{Tutte-2}] In this proof it will be more convenient to let $G$ be a multigraph. Note that we have
\begin{equation}
Z(G,q,w)=wZ(G/e,q,w)+Z(G-e,q,w)
\end{equation}
since we can decompose the sets $F$ according to the cases whether $F$ contains $e$ or not. 

Now we prove the statement by induction on the number of edges. For the empty graph the statement is clearly true. Let $e$ be an edge of $G$, let $f$ be the corresponding edge in another copy of $G$, and with a slight abuse of notation let $e$ and $f$ be the corresponding edges in another $2$-lift $H$ of $G$. Then
\begin{align}
Z(G\cup G,q,w) &= Z((G-e)\cup (G-f),q,w)+wZ((G-e)\cup (G/f),q,w)+\\
&+wZ((G/e)\cup (G-f),q,w)+w^2Z((G/e)\cup (G/f),q,w),
\end{align}
and similarly,
\begin{align}
Z(H,q,w) &=Z(H-\{e,f\},q,w)+wZ((H-e)/f,q,w)+\\
&+wZ((H/e)-f,q,w)+w^2Z(H/\{e,f\},q,w).
\end{align}
By induction we have
\begin{equation}
Z((G-e)\cup (G-f),q,w)\geq Z(H-\{e,f\},q,w).
\end{equation}
Observe that $H/\{e,f\}$ is $2$-lift of $G/e$ so by induction we have
\begin{equation}
w^2Z((G/e)\cup (G/f),q,w)\geq w^2Z(H/\{e,f\},q,w).
\end{equation}
Finally,
\begin{align}
Z((G-e)\cup (G-f),q,w)Z((G/e)\cup (G/f),q,w) &= Z((G-e)\cup (G/f),q,w)^2=\\
&= Z((G/e)\cup (G-f),q,w)^2,
\end{align}
whereas
\begin{align}
Z(H-\{e,f\},q,w)Z(H/\{e,f\},q,w)&\geq Z((H-e)/f,q,w)Z((H/e)-f,q,w)=\\
&= Z((H-e)/f,q,w)^2= \\
&= Z((H/e)-f,q,w)^2.
\end{align}
Here the inequality comes from the FKG-inequality for the random-cluster model, see inequality~\ref{pos-cor}.
This means that 
\begin{align}
Z((G-e)\cup (G/f),q,w)&=(Z((G-e)\cup (G-f),q,w)Z((G/e)\cup (G/f),q,w))^{1/2}\geq \\
&\geq (Z(H-\{e,f\},q,w)Z(H/\{e,f\},q,w))^{1/2}=\\
&= Z((H/e)-f,q,w).
\end{align}
Hence
\begin{equation}
Z(G\cup G,q,w)\geq Z(H,q,w).
\end{equation}

\end{proof}

\subsection{Ruozzi's ideas}

In this section we give a very brief account into the work of N. Ruozzi \cite{Ruo,Ruo2}. Ruozzi investigated a slightly more general setup, the so-called graphical model. 
\medskip

Let $f:\{0,1\}^n\to \mathbb{R}_{\geq 0}$ be a non-negative function. We say that $f$ factors with respect to a hypergraph $G=(V,\mathcal{H})$, where $\mathcal{H}\subseteq 2^V$ if there exist potential functions $\phi_u:\{0,1\}\to \mathbb{R}_{\geq 0}$ for each $u\in V(G)$ and $\psi_{\alpha}:
\{0,1\}^{\alpha}\to \mathbb{R}_{\geq 0}$ for each $\alpha\in \mathcal{H}$ such that
\begin{align} \label{fact}
f(\underline{x})=\prod_{u\in V}\phi_u(x_u)\prod_{\alpha\in \mathcal{H}}\psi_{\alpha}(\underline{x}_{\alpha}),
\end{align}
where $\underline{x}_{\alpha}$ is the subvector of the vector indexed by the set $\alpha$. Finally, let
\begin{align}
Z(G)=\sum_{\underline{x}\in \{0,1\}^n}f(x).
\end{align}

For instance, if $\phi\equiv \nu$ for all $u\in V(G)$ and every $\alpha\in \mathcal{H}$ has size $2$ and for all $\{u,v\}\in \mathcal{H}$ we have $\psi_{u,v}(i,j)=a_{i,j}$ for some matrix $A$ of size $2$ then we get $Z(G)=Z(G;A,\nu)$. 

\begin{Def}
A function $f:\{0,1\}^n\to \mathbb{R}_{\geq 0}$ is \emph{log-supermodular} if for all $\underline{x},\underline{y}\in \{0,1\}^n$ we have
\begin{align}
f(\underline{x})f(\underline{y})\leq f(\underline{x}\wedge \underline{y})f(\underline{x} \vee \underline{y}),
\end{align}
where $(\underline{x}\wedge \underline{y})_i=\min(x_i,y_i)$ and $(\underline{x}\vee \underline{y})_i=\max(x_i,y_i)$. Similarly, A function $f:\{0,1\}^n\to \mathbb{R}_{\geq 0}$ is \emph{log-submodular} if for all $\underline{x},\underline{y}\in \{0,1\}^n$ we have
\begin{align}
f(\underline{x})f(\underline{y})\geq f(\underline{x}\wedge \underline{y})f(\underline{x} \vee \underline{y}).
\end{align}
\end{Def}

\begin{Def}
A factorization of a function $f:\{0,1\}^n\to \mathbb{R}_{\geq 0}$ over $G=(V,\mathcal{H})$ is log-supermodular if for all $\alpha\in \mathcal{H}$, $\psi_{\alpha}(\underline{x}_{\alpha})$ is log-supermodular.
\end{Def}

It turns out that if a function $f$ admits a log-spermodular factorization then the function $f$ itself is log-supermodular. In the graph case when all $\alpha$ has size $2$ then it simply means that the  matrices of size $2\times 2$ corresponding to $\psi_{\alpha}$ have positive determinants. 

Finally, in this more general setting we need to consider the $k$-lift of the function $f$ arising in the form \ref{fact}. Let us consider $k$ copies of each vertex $i$ and let us consider $k$-lifts $\alpha_1,\dots ,\alpha_k$ of $\alpha$ as follows: each $\alpha_i$ contains exactly one copy of vertex $u$ for all $u\in \alpha$. The collection of these new $\alpha_i$'s will be denoted by $\mathcal{H}_k$, the vertex set will be denoted by $V_k$. Let $\psi_{\alpha_1}=\psi_{\alpha_2}=\dots =\psi_{\alpha_k}$, and $\phi_{v'}=\phi_v$ for each copy $v'$ of the vertex $v$. Let $H$ be the corresponding graphical model. We will refer to it as the $k$-lift of $G$.

With all this preparation we are ready to phrase Ruozzi's theorem \cite{Ruo}.

\begin{Th}[N. Ruozzi \cite{Ruo}] \label{Ruo-gen} If $f:\{0,1\}^{V(G)}\to \mathbb{R}_{\geq 0}$ admits a log-supermodular factorization over $G=(V,\mathcal{H})$, then for any $k$-cover $H=(V_k,\mathcal{H}_k)$ of $G$ we have $Z(H)\leq Z(G)^k$.
\end{Th}

If we choose all $\psi_{u,v}=A_{\mathrm{Is}(\beta)}$ for all $(u,v)\in E(G)$ and $\beta>0$ then it immediately implies that for any $k$-lift $H$ of $G$ we have
\begin{align}
Z(H,A_{\mathrm{Is}(\beta)})\leq Z(G,A_{\mathrm{Is}(\beta)})^k
\end{align}
With a little trick Ruozzi was also able to use his theorem to prove that for a bipartite graph $G$ and its $k$-lift $H$ we have
\begin{align}
Z(H,A_{\mathrm{ind}})\leq Z(G,A_{\mathrm{ind}})^k
\end{align}
Finally, for the Potts-model partition function he proved in \cite{Ruo2} that for any graph $G$ and its $k$-lift $H$ we have
\begin{align}
Z(H,q,w)\leq Z(G,q,w)^k
\end{align}
for $q\geq 1,w\geq 0$. Just like in our proof Ruozzi switched to the edge set of the graph $G$ to use a variant of his theorem. 
\medskip

Theorem~\ref{Ruo-gen} is very powerful, we will come back to its applications in Section~\ref{limit}. Here we give a variant of Ruozzi's theorem for $2$-lifts which does not require log-supermodular  factorization, only log-supermodularity and whose proof is simpler. Unfortunately, this proof only works for $2$-lifts.

\begin{Th} Suppose that $f:\{0,1\}^{V(G)}\to \mathbb{R}_{\geq 0}$ is log-supermodular function, and it is the partition function of the graphical model $G=(V,\mathcal{H})$. Then for any $2$-cover $H=(V_k,\mathcal{H}_k)$ of $G$ we have $Z(H)\leq Z(G)^2$.

\end{Th} 

\begin{proof} Let $u$ be a vertex of $G$, and let $u'$ be its pair in the lifts $G\cup G$ and $H$. For $i,j\in \{0,1\}^2$ consider the following quantities:
$$Z_{ij}(G\cup G)=\sum_{\underline{x}\in \{0,1\}^n \atop x_u=i, x_{u'}=j}f(G\cup G,x) \ \ \ Z_{ij}(H)=\sum_{\underline{x}\in \{0,1\}^n \atop x_u=i, x_{u'}=j}f(H,x).$$
Note that $Z_{00}(G\cup G),Z_{11}(G\cup G),Z_{00}(H),Z_{11}(H)$ can be considered as the partition function of $2$-lifts of $G-u$ by simply replacing those $\alpha$ which contains $u$ by $\alpha-u$ and defining $\psi_{\alpha-u}(\underline{x}_{\alpha-u})$ to be $\psi_{\alpha}(\underline{x}_{\alpha})$, where $x_u$ is replaced by $i$ according to which $Z_{ii}$ we consider. By induction we get that
$$Z_{00}(G\cup G)\geq Z_{00}(H)\ \ \ \mbox{and}\ \ \  Z_{11}(G\cup G)\geq Z_{11}(H).$$
Note that
$$Z_{01}(G\cup G)=Z_{10}(G\cup G)=\sqrt{Z_{00}(G\cup G)Z_{11}(G\cup G)}.$$
On the other hand, we have
$$Z_{01}(H)=Z_{10}(H)\leq \sqrt{Z_{00}(H)Z_{11}(H)}.$$
This is true since if $g$ is a log-supermodular function on $\{0,1\}^n$ then for any $k\leq n$ the function $h:\{0,1\}^k\to \mathbb{R}$ defined as follows 
$$h(\underline{y})=\sum_{\underline{x}\in \{0,1\}^{n-k}}g(\underline{y},\underline{x})$$
is also log-supermodular. Hence
$$Z_{01}(H)Z_{10}(H)\leq Z_{00}(H)Z_{11}(H).$$
Now 
\begin{align}
Z(H)&=Z_{00}(H)+Z_{11}(H)+Z_{01}(H)+Z_{10}(H) \\
&\leq Z_{00}(H)+Z_{11}(H)+2\sqrt{Z_{00}(H)Z_{11}(H)} \\
&\leq Z_{00}(G\cup G)+Z_{11}(G\cup G)+2\sqrt{Z_{00}(G\cup G)Z_{11}(G\cup G)}\\
&=Z_{00}(G\cup G)+Z_{11}(G\cup G)+Z_{01}(G\cup G)+Z_{10}(G\cup G)\\
&=Z(G\cup G).
\end{align}
Since $Z(G\cup G)=Z(G)^2$ we are done.

\end{proof}

\begin{Rem} N. Ruozzi informed us (personal communication) that his result  (and also the above theorem) implies many results in this paper, in particular Theorem~\ref{pos}, since many homomorphism functions can be reduced to a log-supermodular function with a clever trick. We do not detail this trick since N. Ruozzi may wish to publish his idea in a forthcoming paper.

\end{Rem}

\section{Zhao's and Sernau's ideas} \label{other work}

In this section we relate our work with some previous work, most notably due to Y. Zhao and L. Sernau.

\subsection{Bipartite swapping target graphs}

In this section we clarify what is the connection between our results and the so called bipartite swapping target graphs introduced by Y. Zhao \cite{Zhao2}. The definition we use for bipartite swapping target graphs is actually Proposition 4.6 in \cite{Zhao2}.

\begin{Def} From a graph $H$ let us define the the graph $H^{bst}$ as follows:
$V(H^{bst})=V(H)\times V(H)$, and there is an edge between $(u,v)$ and $(u',v')\in V(H^{bst})$ if $(u,u')\in E(H)$, $(v,v')\in  V(H^{bst})$, and $(u',v)\notin E(H)$ or  $(u,v')\notin E(H)$. Then we say that $H$ is a \emph{bipartite swapping target graph} if $H^{bst}$ is bipartite.
\end{Def}

Y. Zhao \cite{Zhao2} showed that for a bipartite swapping target graph $H$ we have
\begin{align}
\hom(G,H)^2\leq \hom(G\times K_2,H)
\end{align}
for any graph $G$. It is very natural to define the sibling of this concept.

\begin{Def} From a graph $H$ let us define the the graph $H^{abst}$ as follows:
$V(H^{abst})=V(H)\times V(H)$, and there is an edge between $(u,v)$ and $(u',v')\in V(H^{abst})$ if $(u,v')\in E(H)$, $(u',v)\in  V(H^{abst})$, and $(u,u')\notin E(H)$ or  $(v,v')\notin E(H)$. Then we say that $H$ is a \emph{adjoint bipartite swapping target graph} if $H^{abst}$ is bipartite.
\end{Def}

It is not hard to modify Y. Zhao's argument to show that for an adjoint  bipartite swapping target graph $H$ we have
\begin{align}
\hom(G,H)^2\geq \hom(G\times K_2,H)
\end{align}
for any graph $G$.
\medskip

Now let $A$ be the adjacency matrix of $H$, this is a $0-1$ matrix as we assume that $H$ is a simple graph possibly with loops. The condition for $E(H^{bst})$ saying that 
"$(u,u')\in E(H)$, $(v,v')\in  V(H^{bst})$, and $(u',v)\notin E(H)$ or  $(u,v')\notin E(H)$" means that $A^=$ contains a $1$ at the entry $((u,v),(u',v'))$, but it is $0$ in $A^{\times}$ at the same entry. Similarly, the condition for $E(H^{abst})$ saying that  "$(u,v')\in E(H)$, $(u',v)\in  V(H^{abst})$, and $(u,u')\notin E(H)$ or  $(v,v')\notin E(H)$" means that $A^{\times}$ contains a $1$ at the entry $((u,v),(u',v'))$, but it is $0$ in $A^{=}$ at the same entry. In other words, for the matrix $D=\frac{1}{2}\left(A^=-A^{\times}\right)$ we have $D((u,v),(u',v'))=1$ if and only if $((u,v),(u',v'))\in E(H^{bst})$, and $D((u,v),(u',v'))=-1$ if and only if $((u,v),(u',v'))\in E(H^{abst})$.

Observe that the "diagonal" vertices $(u,u)$ are isolated vertices in both $H^{bst}$ and $H^{abst}$. 

This means that the graphs covered by part (b) of Theorem~\ref{gen} are all bipartite swapping target graphs, and the graphs covered by part (a) of Theorem~\ref{gen} are all adjoint bipartite swapping target graphs. Of course, since we also covered matrices not just graphs our result is slightly more general in the sense that it is not clear how to interpret the corresponding concept "bipartite swapping target matrix". Unfortunately, there are graphs which are bipartite swapping target graphs, but are not covered by part (b) of Theorem~\ref{gen}. Note that if $A$ is bipartite graph then it is a bipartite swapping target graph, and the corresponding result is trivial:
\begin{align}
\hom(G,A)^2\leq \hom(G\times K_2,A)
\end{align}
since if $G$ is not bipartite then $\hom(G,A)=0$, and if $G$ is bipartite then $G\cup G=G\times K_2$ and consequently 
\begin{align}
\hom(G,A)^2=\hom(G\cup G,A)=\hom(G\times K_2,A).
\end{align}
On the other hand, an inequality of type
\begin{align}
\hom(G,A)^2\geq \hom(H,A)
\end{align}
for any bipartite graph $G$ and its $2$-lift $H$ would be very non-trivial statement. One might naively hope that a simple modification of Y. Zhao's proof works in this more general setting, but it is not true. His proof and also our proof of Theorem~\ref{gen} is based on the idea that if we have a homomorphism of $H$ to a graph $A$ then if we consider the pairs projected to every vertices of the original graph $G$ then we can lift it back to get a homomorphism of $G\times K_2$ (or $G\cup G$). Unfortunately, this is not always true even if $A$ is bipartite. Let $G=C_4$, the cycle on $4$ vertices. Then $G\cup G=G\times K_2=C_4\cup C_4$. Let $H=C_8$. Let us consider the following $3$-coloring which can be considered as a homomorphism into $C_6$. There is no proper lift of it to $C_4\cup C_4$, but it can be lifted to $H=C_8$.

\begin{figure}[h!]
\scalebox{.45}{\includegraphics{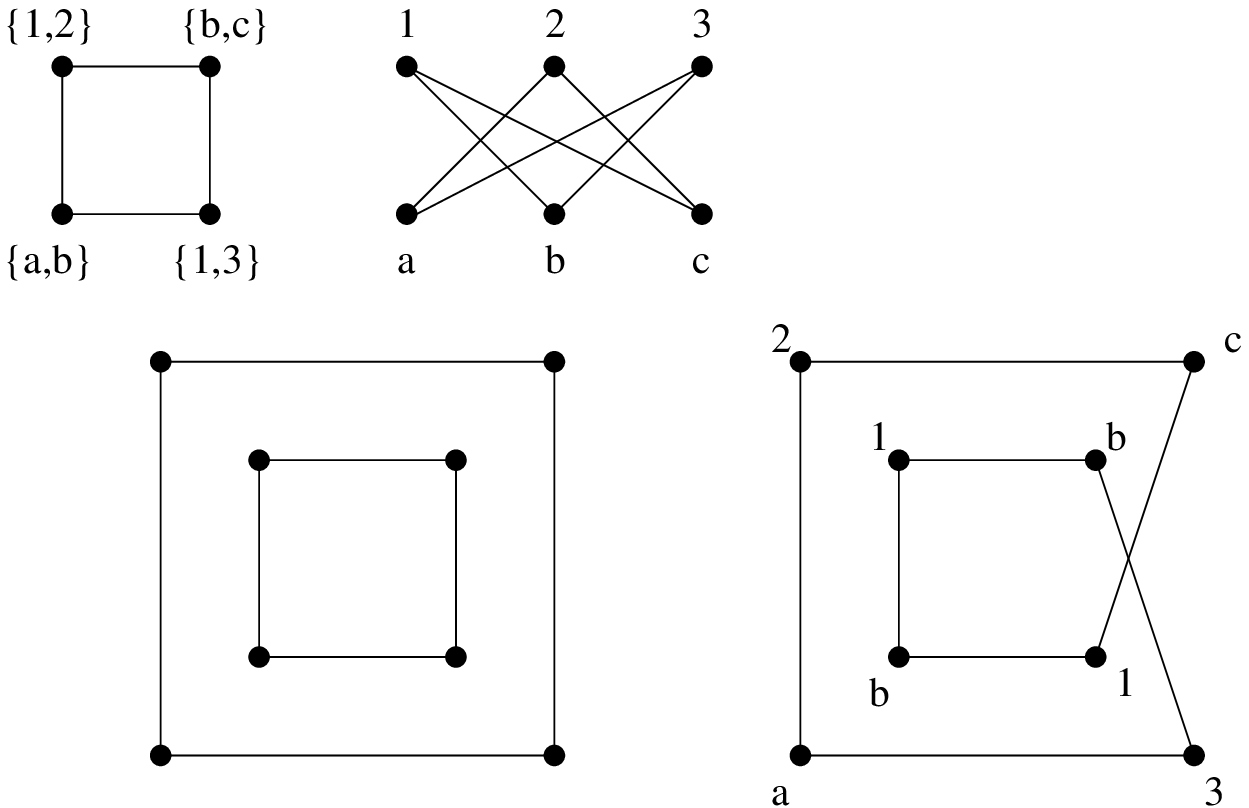}}  
\end{figure}

Naturally, $\hom(C_8,C_6)<\hom(C_4\cup C_4,C_6)$, but there does not seem to be a natural injection from the set of homomorphisms from $C_8$ to $C_6$ to the set of homomorphisms from $C_4\cup C_4$ to $C_6$.

\subsection{Sernau's ideas} L. Sernau \cite{S} introduced a series of ideas to prove inequalities of type 
\begin{align}
\hom(G,A)^2\leq \hom(G\times K_2,A).
\end{align}
These ideas were based on various graph transformations. Here we list some of them.

\begin{itemize}
\item Tensor product $G\times H$: its vertices are $V(G)\times V(H)$, with $(u,v)$ and $(u',v')\in V(G)\times V(H)$ adjacent in $G\times H$ if $(u,u')\in E(G)$ and $(v,v')\in E(H)$. This construction is also called categorical product. (If $A_1$ and $A_2$ are matrices then we keep the notation $A_1\otimes A_2$ in spite of the fact that two concepts are completely analogous: if $A(G)$ denotes the adjacency matrix of $G$, then $A(G\times H)=A(G)\otimes A(H)$.)

\item Exponentiation $H^G$: its vertices are the maps $f:V(G)\to V(H)$ (not necessarily homomorphisms), and $(f,f')\in E(H^G)$ if and only if $(f(u),f(u'))\in E(H)$ whenever $(u,u')\in E(G)$.

\item $G^{\circ}$ is the graph obtained from $G$ by adding a loop at each vertex of $G$.

\item $\ell(G)$ is the induced subgraph of $G$ induced by those vertices which have a loop.

\item $\mathrm{Sub}(G)$ is the subdivision of $G$: we subdivide each edge by a new vertex. So the obtained graph is a bipartite graph with $|V(G)|+|E(G)|$ vertices. 

\end{itemize}

The following identities are easy, but very useful.
\begin{equation} \label{tensor}
\hom(G,H_1\times H_2)=\hom(G,H_1)\hom(G,H_2).
\end{equation}
\begin{equation} \label{tensor2}
Z(G,A_1\otimes A_2)=Z(G,A_1)Z(G,A_2).
\end{equation}
\begin{equation} \label{exp}
\hom(G\times G',H)=\hom(G,H^{G'}).
\end{equation}
\begin{equation} \label{loop}
\hom(G^{\circ},H)=\hom(G,\ell(H)).
\end{equation}
\begin{equation} \label{sub}
Z(\mathrm{Sub}(G),A)=Z(G,A^2).
\end{equation}

\begin{Th}
(i) If $A_1,A_2\in \mathcal{A}$ (resp. $\mathcal{A}^b$ or $\mathcal{B}$) then $A_1\times A_2\in \mathcal{A}$ (resp. $\mathcal{A}^b$ or $\mathcal{B}$). \\
(ii) If $F\in \mathcal{A}^b$ and $B$ is bipartite then $F^B,\ell(F^B)\in \mathcal{A}$. \\
(iii) If $A\in \mathcal{A}^b$ then $A^2\in \mathcal{A}$.
\end{Th}

\begin{proof} Part (i) immediately follows from \ref{tensor} and \ref{tensor2}. 

To prove part (ii) we need the following observations: $(G\cup G)\times B=(G\times B)\cup (G\times B)$, where $G\times B$ is bipartite since $B$ is bipartite, and if $H$ is a $2$-lift of $G$ then $H\times B$ is a $2$-lift of $G\times B$.
\begin{align}
\hom(G\cup G,F^B) &=\hom((G\cup G)\times B,F)=\\
&=\hom((G\times B)\cup (G\times B),F)\geq \\
&\geq \hom(H\times B,F)=\hom(H,F^B).
\end{align}
Similarly, 
\begin{align}
\hom(G\cup G,\ell(F^B)) &=\hom((G\cup G)^{\circ},F^B)=\\
&=\hom((G\cup G)^{\circ}\times B,F)= \\
&=\hom(G^{\circ}\times B\cup G^{\circ}\times B,F)\geq \\
&\geq \hom(H^{\circ}\times B,F)=\\
&=\hom(H^{\circ},F^B)=\\
&=\hom(H,\ell(F^B)).
\end{align}

To prove part (iii) we need the following very easy observations: we have 
$\mathrm{Sub}(G\cup G)=\mathrm{Sub}(G)\cup \mathrm{Sub}(G)$,
and if $H$ is a $2$-lift of $G$ then $\mathrm{Sub}(H)$ is $2$-lift of $\mathrm{Sub}(G)$. Hence
if $A\in \mathcal{A}^b$ then
\begin{align}
Z(G\cup G,A^2)=Z(\mathrm{Sub}(G\cup G),A)=Z(\mathrm{Sub}(G)\cup \mathrm{Sub}(G),A)\geq Z(\mathrm{Sub}(H),A)=Z(H,A^2).
\end{align}
At the inequality we used the fact that $\mathrm{Sub}(G)$ is a bipartite graph.

\end{proof}

We know that $A_{\mathrm{ind}}\in \mathcal{B}\subseteq \mathcal{A}^b$ so by (ii) we have $\ell(A_{\mathrm{ind}}^{K_2})\in \mathcal{A}$. As it was pointed out in \cite{CCPT,S} we have $\ell(A_{\mathrm{ind}}^{K_2})=A_{\mathrm{WR}}$, so Theorem~\ref{Th-B} implies Theorem~\ref{Th-A}. We also remark that $A=A_{\mathrm{ind}}\times A_{\mathrm{ind}}$ is not a $TP_2$ or $TN_2$ matrix so this matrix was not covered by any of the previous theorems. 
\bigskip

L. Sernau \cite{S} also showed $d$-regular graphs and target graph $A$ for which
\begin{align}
\hom(G,A)^{1/v(G)}>\max(\hom(K_{d+1},A)^{1/v(K_{d+1})},\hom(K_{d,d},A)^{1/v(K_{d,d})})
\end{align}
if $d>3$. (One can extend his counter examples to the case $d=3$ too.)
An analogous problem is the following: is it true that for any graph $A$ and a $2$-lift $H$ of $G$ we have
\begin{align}
\hom(H,A)\leq \max(\hom(G\cup G,A),\hom(G\times K_2,A))?
\end{align}
It turns out that the answer is  negative. On the other hand, we do not know whether the answer is negative if $G$ is bipartite.

\section{The limit value and Sidorenko's conjecture} \label{limit}

In this section we give a very very brief account into the work of A. Dembo, A. Montanari, A. Sly and N. Sun \cite{DM,DM2,DMS,DMSS,SlSu,SlSu2} and the work of P. Vontobel \cite{Von}.
\medskip

Following Vontobel \cite{Von}, let us consider the following quantity for a graph $G$.
\begin{align}
\Phi_B(G;A,\nu)=\lim_{k\to \infty}\frac{1}{kv(G)}\ln \mathbb{E} Z(G^k,A,\nu),
\end{align}
where the expectation is taken for all $k$-covers of the graph $G$. This is called the normalized Bethe-partition function of the graph $G$.
We have seen that Ruozzi proved (cf. \cite{Ruo,Ruo2}) that in certain cases we have
\begin{align}
Z(G,A,\nu)^k\geq Z(G^k,A,\nu)
\end{align}
for all $k$-lifts of the graph $G$. This immediately implies that 
\begin{align}
\frac{1}{v(G)}\ln Z(G,A,\nu)\geq \Phi_B(G;A,\nu)
\end{align}
P. Vontobel proved in \cite{Von} that $\Phi_B(G;A,\nu)$ can be defined through an optimization problem on the so-called local marginal polytope of $G$. 
The local marginal polytope $\mathcal{T}(G)$ is defined as follows. In what follows $\tau_{u,v}$ is a probability distribution on $[q]^2$ for every $(u,v)\in E(G)$ and  $\tau_{u}$ is a probability distribution on $[q]$ for every $u\in V(G)$.
\begin{align}
\mathcal{T}(G)=\{\tau\geq 0\ |\ \forall (u,v)\in E(G):\ \sum_{j\in [q]}\tau_{u,v}(i,j)=\tau_u(i)\ \mbox{and}\ \forall u\in V(G):\  \sum_{i\in [q]}\tau_u(i)=1\}
\end{align}
For $\tau\in \mathcal{T}(G)$ let
\begin{align}
\Phi_B(G,\tau;A,\nu)=\frac{1}{v(G)}(U(G,\tau;A,\nu)-H(G,\tau;A,\nu)),
\end{align}
where
\begin{align}
U(G,,\tau;A,\nu)=\sum_{u\in V(G)}\sum_{i\in [q]}\tau_u(i)\ln \nu(i)+\sum_{(u,v)\in E(G)}\sum_{(i,j)\in {q}\times [q]}\tau_{u,v}(i,j)\ln a_{i,j}
\end{align}
and
\begin{align}
H(G,\tau;A,\nu)=\sum_{u\in V(G)}\sum_{i\in [q]}\tau_u(i)\ln \tau_u(i)+\sum_{(u,v)\in E(G)}\sum_{(i,j)\in {q}\times [q]}\tau_{u,v}(i,j)\ln \frac{\tau_{u,v}(i,j)}{\tau_u(i)\tau_v(j)}.
\end{align}
Finally, let
\begin{align}
\Phi_B(G;A,\nu)=\max_{\tau\in \mathcal{T}(G)}\Phi_B(G,\tau;A,\nu)
\end{align}

Now let us consider what happens if we suppose that $\tau_{u,v}$ distribution coincides with some distribution $h$ for all edge $(u,v)$ of a $d$--regular graph $G$.

Let $\mathcal{H}$ be the set of probability distributions on the pairs $(i,j)$, where $i,j\in [q]$ such that $h(i,j)=h(j,i)$ for all $i,j\in [q]$. Let $\bar{h}$ be the one-point marginal of $h$. Let us  fix a non-negative symmetric matrix $A$ and a positive weighting $\nu$. In the rest of this section we assume that $A$ is permissive, i. e., there exists an $i\in [q]$ such that $a_{ij}>0$ for all $j\in [q]$.

Let
\begin{equation}
\Phi_d(A,\nu; h)=\sum_{i\in [q]}\bar{h}(i)\ln \nu(i)-(d-1)H(\bar{h})+\frac{d}{2}\left(H(h)+\sum_{i,j\in [q]}h(i,j)\ln a_{i,j}\right),
\end{equation}
where $H$ is the entropy function: $H(p)=-\sum_{i\in [q]}p_i\ln (p_i)$. If $a_{ij}=0$ then $h(i,j)\ln a_{ij}=0$ if $h(i,j)=0$, and otherwise it is $-\infty$.

Let
\begin{equation} \label{def1}
\Phi_d(A,\nu)=\sup_{h\in \mathcal{H}}\Phi(A,\nu; h).
\end{equation}

Alternatively, one can define $\Phi(A,\nu)$ through the Belief Propagation and Bethe prediction (see Definition 1.3 and 1.5 in \cite{SlSu2}) as follows. For a probability distribution $\widetilde{h}$ on $[q]$ one can define the following probability distribution:
\begin{equation} \label{BP}
\mathrm{BP}\widetilde{h}(i)=\frac{1}{z_{\widetilde{h}}}\nu(i)\left(\sum_{j\in [q]}a_{ij}\widetilde{h}(j)\right)^{d-1}.
\end{equation} 
Let $\mathcal{H}^*$ be the set of BP fixed points. For a probability distribution $\widetilde{h}$ on $[q]$ let
\begin{equation}
\widetilde{\Phi}_d(A,\nu; \widetilde{h})=\ln \left(\sum_{i\in [q]}\nu(i)\left(\sum_{j \in [q]}a_{ij}\widetilde{h}(j)\right)^d\right)-\frac{d}{2}\ln \left(\sum_{i,j\in [q]}a_{ij}\widetilde{h}(i)\widetilde{h}(j)\right).
\end{equation}
Then
\begin{equation}  \label{def2}
\Phi_d(A,\nu)=\sup_{\widetilde{h}\in \mathcal{H}^*}\widetilde{\Phi}(A,\nu; \tilde{h}).
\end{equation}
The connection between the two definitions, \ref{def1} and \ref{def2}, of $\Phi(A,\nu)$ is the following. If $h$ maximizes $\Phi(A,\nu; h)$ then 
\begin{equation}
h(i,j)=\frac{1}{\widetilde{S}}a_{ij}\widetilde{h}(i)\widetilde{h}(j),
\end{equation}
for some $\widetilde{h}(i)\in \mathcal{H}^*$, and normalizing constant $\widetilde{S}$. This way the two definitions lead to the same value $\Phi_d(A,\nu)$, for details see Proposition 1.7 in \cite{DMSS} or Theorem 1.18 in \cite{DMS}.
\bigskip

A. Dembo, A. Montanari, A. Sly and N. Sun \cite{DMSS} proved that if we take a random $d$--regular graph $G_n$ on $n$ vertices then
\begin{equation}
\lim_{n\to \infty}\frac{1}{n}\ln \mathbb{E}_nZ(G_n,A,\nu)=\Phi_d(A,\nu).
\end{equation}
This can be considered as a special case of Vontobel's result applied to one vertex graph with $d$ loops, since the $n$ lifts of this graph are exactly the $d$--regular graphs.

From the above discussion it is clear that for every $d$--regular graph $G$ we have
\begin{align}
\Phi_B(G;A,\nu)\geq \Phi_d(A,\nu)
\end{align}
In particular, if we consider the graph $K_2(d)$ consisting of two vertices and $d$ parallel edges between them then 
\begin{align}
\Phi_B^b(A,\nu):=\Phi_B(K_2(d);A,\nu)\geq \Phi_d(A,\nu)
\end{align}
Note that the $n$-lifts of $K_2(d)$ are the $d$--regular bipartite graphs on $2n$ vertices.
Let us define
\begin{equation}
\phi_d(A,\nu)=\liminf_{n\to \infty}\frac{1}{n}\mathbb{E}_n\ln Z(G_n,A,\nu).
\end{equation}
By Jensen's inequality we have
\begin{equation}
\mathbb{E}_n\ln Z(G_n,A,\nu)\leq \ln \mathbb{E}_nZ(G_n,A,\nu)
\end{equation}
implying that
\begin{equation}
\phi_d(A,\nu)\leq \Phi_d(A,\nu).
\end{equation}
Finally, let
\begin{equation}
\phi_d^m(A,\nu)=\inf_{G\in \mathcal{G}_d} \frac{1}{v(G)}\ln Z(G,A,\nu),
\end{equation}
and 
\begin{equation}
\phi^{b,m}_d(A,\nu)=\inf_{G\in \mathcal{G}^b_d} \frac{1}{v(G)}\ln Z(G,A,\nu),
\end{equation}
Clearly,
\begin{equation}
\phi^m_d(A,\nu)\leq \phi_d(A,\nu)\leq \Phi_d(A,\nu)
\end{equation}
 and 
\begin{equation}
\phi^{b,m}_d(A,\nu)\leq \Phi^b_d(A,\nu).
\end{equation}
It is known that it can occur that $\phi_d(A,\nu)<\Phi_d(A,\nu)$. For instance for $A=A_{\mathrm{ind}}$ and $\nu_{\lambda}=(1,\lambda)$ we get that
$\phi_d(A_{\mathrm{ind}},\nu_{\lambda})<\Phi_d(A_{\mathrm{ind}},\nu_{\lambda})$ if $\lambda>\frac{(d-1)^{d-1}}{(d-2)^d}$. The surprising fact that in many cases it is still true that $\phi^m_d(A,\nu)=\phi_d(A,\nu)=\Phi_d(A,\nu)$ or $\phi^{b,m}_d(A,\nu)=\Phi_d(A,\nu)$.

For instance, it was proved by A. Dembo and A. Montanari \cite{DM,DM2} that if $(G_i)$ is a sequence of $d$--regular graphs with $g(G_i)\to \infty$, $A=A_{\mathrm{Is}(\beta)}$, $\nu_B=(e^B,e^{-B})$ then we have
\begin{equation}
\lim_{i\to \infty}\frac{1}{v(G_i)}\ln Z(G,A_{\mathrm{Is}(\beta)},\nu_B)=\Phi_d(A_{\mathrm{Is}(\beta)},\nu_B)
\end{equation}
if $\beta\geq 0$. Combining it with Theorem~\ref{pos} we immediately get that $\phi^m_d(A_{\mathrm{Is}(\beta)},\nu_B)=\Phi_d(A_{\mathrm{Is}(\beta)},\nu_B)$.
N. Sun and A. Sly \cite{SlSu,SlSu2} also proved that if $(G_i)$ is a sequence of $d$--regular bipartite graphs with $g(G_i)\to \infty$ then we have
\begin{equation}
\lim_{i\to \infty}\frac{1}{v(G_i)}\ln Z(G,A_{\mathrm{Is}(\beta)},\nu_B)=\Phi(A_{\mathrm{Is}(\beta)},\nu_B)
\end{equation}
even if $\beta<0$. Combining it with Theorem~\ref{neg} this shows that for all $\beta$ we have $\phi^{b,m}_d(A_{\mathrm{Is}(\beta)},\nu_B)=\Phi_d(A_{\mathrm{Is}(\beta)},\nu_B)$. In the same paper N. Sun and A. Sly \cite{SlSu,SlSu2} also proved that $(G_i)$ is a sequence of $d$--regular bipartite graphs with $g(G_i)\to \infty$, $A=A_{\mathrm{ind}}$, $\nu_{\lambda}=(1,\lambda)$ then we have
\begin{equation}
\lim_{i\to \infty}\frac{1}{v(G_i)}\ln Z(G,A_{\mathrm{ind}},\nu_{\lambda})=\Phi(A_{\mathrm{ind}},\nu_{\lambda})
\end{equation}
for all $\lambda\geq 0$. Combining it with Theorem~\ref{neg} this again shows that $\phi^{b,m}_d(A_{\mathrm{ind}},\nu_{\lambda})=\Phi_d(A_{\mathrm{ind}},\nu_{\lambda})$ for all $\lambda\geq 0$.

\subsection{Sidorenko's conjecture}

Sidorenko's conjecture states that for a bipartite graph $G$ and a graph $A$ on $q$ vertices we have
\begin{equation}
\hom(G,A)=q^{v(G)}\left(\frac{\hom(K_2,A)}{q^2}\right)^{e(G)}.
\end{equation}
Clearly, the natural weighted version for a pair $(A,\nu)$ is 
\begin{equation}
Z(G,A,\nu)\geq \left(\sum_{i\in [q]}\nu(i)\right)^{v(G)}\left(\frac{\sum_{i,j}\nu(i)\nu(j)a_{ij}}{(\sum_{i\in [q]}\nu(i))^2}\right)^{e(G)}.
\end{equation}
If $G$ is $d$--regular this is equivalent with the inequality
\begin{equation}
\frac{1}{v(G)}\ln Z(G,A,\nu)\geq \ln \left(\sum_{i\in [q]}\nu(i)\right)+\frac{d}{2}\ln \left(\frac{\sum_{i,j}\nu(i)\nu(j)a_{ij}}{(\sum_{i\in [q]}\nu(i))^2}\right).
\end{equation}
Let
\begin{equation}
S_d(A,\nu)=\ln  \left(\sum_{i\in [q]}\nu(i)\right)+\frac{d}{2}\ln \left(\frac{\sum_{i,j}\nu(i)\nu(j)a_{ij}}{(\sum_{i\in [q]}\nu(i))^2}\right).
\end{equation}
One can check that
\begin{equation}
\Phi_d(A,\nu)\geq S_d(A,\nu).
\end{equation}
Indeed, let
\begin{equation}
h(i,j)=\frac{\nu(i)\nu(j)a_{ij}}{S},
\end{equation}
where
\begin{equation}
S=\sum_{i,j}\nu(i)\nu(j)a_{ij}.
\end{equation}
For any $j\in [q]$ let
\begin{equation}
\tilde{\nu}(j)=\frac{\nu(j)}{\sum_{i\in [q]}\nu(i)}.
\end{equation}
Then
\begin{equation}
\Phi_d(A,\nu)\geq \Phi_d(A,\nu;h)=S_d(A,\nu)+(d-1)D(\bar{h} || \tilde{\nu})\geq S_d(A,\nu),
\end{equation}
where $D(p||q)=\sum_{i}p(i)\ln \frac{p(i)}{q(i)}$, the Kullback--Leibler distance of probability distributions $p$ and $q$, this is always a non-negative quantity.

\subsection{Case study: the number of independent sets.} \label{case:ind}
As before let $A=A_{\mathrm{ind}}$ and $\nu=(1,\lambda)$. In this case only distributions $h$ with $h_{22}=0$ can maximize $\Phi(A,\nu; h)$. A natural parametrization is $h_{12}=h_{21}=\alpha$ and $h_{11}=1-2\alpha$. Then
$\bar{h}_1=1-\alpha$ and $\bar{h}_2=\alpha$. A small computation shows that the maximizing $\alpha$ satisfies 
\begin{equation}
\frac{\alpha}{\lambda (1-\alpha)}=\left(\frac{1-2\alpha}{1-\alpha}\right)^d,
\end{equation}
and with this $\alpha$ we have
\begin{equation}
\Phi_{\lambda}=\frac{1}{2}\ln \left(\frac{\lambda (1-\alpha)^{d-1}}{\alpha}\right)=\frac{1}{2}\ln \left( \frac{(1-\alpha)^{2(d-1)}}{(1-2\alpha)^d}\right).
\end{equation}
Hence combining it with Theorem~\ref{Th-B} we have the following theorem. This theorem also follows from a result of the paper \cite{Ruo}.

\begin{Th} For any $\lambda\geq 0$ let $\alpha$ be the unique solution of
\begin{align}
\frac{\alpha}{\lambda (1-\alpha)}=\left(\frac{1-2\alpha}{1-\alpha}\right)^d
\end{align}
in the interval $[0,1/2]$.
Let $G$ be a $d$--regular bipartite graph $G$. Let $I(G,\lambda)=\sum_ki_k(G)\lambda^k$, where $i_k(G)$ denotes the number of independent sets of size $k$ in the graph $G$.
Then we have
\begin{align}
I(G,\lambda)\geq \left(\frac{\lambda (1-\alpha)^{d-1}}{\alpha}\right)^{v(G)/2}.
\end{align}
\end{Th}

\subsection{Case study: Ising-model.} \label{case:Ising} Let us consider the weighted case when $\nu_B(1)=e^B$ and $\nu_B(-1)=e^{-B}$. (Note that we use the labels $1,-1$ instead of $1,2$ for the vertices of the target graph (matrix).) In other words, 
\begin{equation}
Z(G,A_{\mathrm{Is(\beta)}},\nu_B)=\sum_{\underline{x}\in \{-1,1\}^{V(G)}}\exp \left( \beta \sum_{(u,v)\in E(G)}x_ux_v+B\sum_{u\in V(G)}x_u\right).
\end{equation}
It turns out (see \cite{DM,DM2}) that when $\beta\geq 0$ and $(G_i)$ is sequence of $d$--regular graphs such that $g(G_i)\to \infty$ then 
\begin{equation}
\lim_{i\to \infty} \frac{1}{v(G_i)}\ln Z(G_n,A_{\mathrm{Is(\beta)}},\nu_B)=\varphi_d(\beta,B),
\end{equation}
where
\begin{align}
  \varphi_d(\beta,B)=\varphi_d(\beta,B,h^*)=
\end{align}
\begin{align}
=\frac{d}{2}\left(-\frac{1}{2}\ln(1-\theta^2)-\ln(1+\theta \tanh^2(h^*))\right)+\ln \left(e^B(1+\theta \tanh(h^*))^d+e^{-B}(1-\theta \tanh(h^*))^d\right),
\end{align}
where $\theta=\tanh(\beta)$ and $h^*$ is the largest solution of the equation
\begin{equation}
h=B+(d-1)\mathrm{atanh}(\theta \tanh(h)).
\end{equation}
Note that this formula is valid even if $B=0$. Combining it with our Theorem~\ref{gen} we get that for any graph $G$ we have
\begin{equation}
\frac{1}{v(G)}\ln Z(G,A_{\mathrm{Is(\beta)}},\nu_B)\geq\varphi_d(\beta,B).
\end{equation}
It might be more enlightening just to write out Sidorenko's inequality in this case:
\begin{equation}
Z(G,A_{\mathrm{Is(\beta)}},\nu_B)\geq \left(e^B+e^{-B}\right)^{v(G)}\left(\frac{e^{\beta}\left(e^{2B}+e^{-2B}\right)+2e^{-\beta}}{\left(e^B+e^{-B}\right)^2}\right)^{e(G)}.
\end{equation}

\subsection{Case study: Potts-model and Tutte-polynomial.} \label{case:Tutte}

In this case it is again true that if $(G_i)$ is sequence of $d$--regular graphs such that $g(G_i)\to \infty$ then 
\begin{equation}
\lim_{i\to \infty} \frac{1}{v(G_i)}\ln Z(G_n,A_q(w))
\end{equation}
exists when $q$ is an integer, $w\geq 0$ and $d$ is even, see  \cite{DMSS,DMS}. Let us mention that the conjectured proper limes infimum is already established in \cite{DMS}, and of course, it is enough for the applications.
 
Instead of giving the exact form of the limit we note that Sidorenko's conjecture is trivial in the case $q\geq 1,w\geq 0$. Indeed, it asserts that
\begin{equation}
Z(G,q,w)\geq q^{v(G)}\left(1+\frac{w}{q}\right)^{e(G)}.
\end{equation}
Note that for any subset $F\subseteq E(G)$ we have $k(F)\geq v(G)-|F|$ by induction on $|F|$. Hence
\begin{equation}
Z(G,q,w)=\sum_{F\subseteq E(G)}q^{k(F)}w^{|F|}\geq \sum_{F\subseteq E(G)}q^{v(G)-|F|}w^{|F|}=q^{v(G)}\left(1+\frac{w}{q}\right)^{e(G)}.
\end{equation}
Clearly, this means that if $G$ is $d$--regular graph then $Z(G,q,w)^{1/v(G)}\geq q \left(1+\frac{w}{q}\right)^{d/2}$. Note that there is another trivial lower bound for $Z(G,q,w)$:
\begin{equation}
Z(G,q,w)=\sum_{F\subseteq E(G)}q^{k(F)}w^{|F|}\geq \sum_{F\subseteq E(G)}w^{|F|}=(1+w)^{e(G)}.
\end{equation}
This shows that for a $d$--regular graph $G$ we have $Z(G,q,w)^{1/v(G)}\geq \left(1+w\right)^{d/2}$.

\subsection{Non-regular graphs and Benjamini--Schramm convergence.} \label{BS-conv}

Since in the applications of $2$-lifts we never used the regularity of the graph, it is possible to use the ideas of this paper for non-regular graphs. For matchings of non-regular graphs such a program was carried out M. Lelarge \cite{Le}. Note that it is still possible to construct for every graph $G$ a sequence of graphs $(G_i)$ such that $G_0=G$, $G_i$ is a $2$-lift of $G_{i-1}$, and $g(G_i)\to \infty$. Then it is a natural question whether there is a limit object in this case too like $\mathbb{T}_d$. The answer is yes: it is the universal cover tree of $G$, more precisely the universal cover tree with the uniform distribution of the lifts of the vertices of the original graph as a root. To make this statement precise we recall the definition of Benjamini--Schramm convergence and random rooted graphs (unimodular random graphs). 

\begin{Def} \label{BS-convergence} Let $L$ be a probability distribution on (infinite) connected rooted graphs; we will call $L$ a \emph{random rooted graph}.
For a finite connected rooted graph $\alpha$ and a positive integer $r$, let $\mathbb{P}(L,\alpha,r)$ be the probability that the $r$-ball
centered at a random root vertex chosen from the distribution $L$ is isomorphic to $\alpha$.

For a finite graph $G$, a finite connected rooted graph $\alpha$ and a positive integer
$r$, let $\mathbb{P}(G,\alpha,r)$ be the probability that the $r$-ball
centered at a uniform random vertex of $G$ is isomorphic to $\alpha$. 

We say that a bounded-degree graph sequence $(G_i)$ is \emph{Benjamini--Schramm
convergent} if for all finite rooted graphs $\alpha$ and $r>0$, the
probabilities $\mathbb{P}(G_i,\alpha,r)$ converge. Furthermore, we say that \emph{$(G_i)$ Benjamini-Schramm converges to $L$},
if for all positive integers $r$ and finite rooted graphs $\alpha$, $\mathbb{P}(G_i,\alpha,r)\rightarrow \mathbb{P}(L,\alpha,r)$.

The Benjamini--Schramm convergence is also called \emph{local convergence} as it primarily grasps the local structure of the graphs $(G_i)$.
\end{Def}

Not every random rooted graph can be the limit of finite graphs. There is an extra condition called unimodularity, for details see \cite{LL}. From the definition of Benjamini--Schramm convergence it is quite straightforward to see that a sequence $(G_i)$ of lifts of $G$ with $g(G_i)\to \infty$ converges to the universal cover tree of $G$, see also \cite{Le}. Fortunately, in many notable cases A. Dembo, A. Montanari and N. Sun \cite{DMS} established the limit theorem even in the non-regular setting. In fact, they proved a much more general theorem covering sequences converging to unimodular random trees.

\section{Concluding remarks and open problems} \label{the end}

\subsection{Concluding remarks.} The goal of this section is to elaborate on an admittedly vague intuition of the author.  It seems that for many interesting graph parameters there is a local-global principle in the following sense. Inequalities between $2$-lifts and correlation inequality can predict the extremal regular graph. The simplest instance is of course that
\begin{align}
\hom(G\cup G,A)\leq \hom(G\times K_2,A)
\end{align}
implies
\begin{align}
\hom(G,A)^{1/v(G)}\leq  \hom(K_{d,d},A)^{1/v(K_{d,d})}
\end{align}
for every $d$--regular graph $G$. And we have seen that 
\begin{align}
\hom(G\cup G,A)\geq \hom(H,A)
\end{align}
implies that
\begin{align}
\hom(G,A)^{1/v(G)}\geq "\hom(\mathbb{T}_d,A)^{1/v(\mathbb{T}_d)}"
\end{align}
for every $d$--regular graph $G$. But the point is that the validity of the very same inequality coincides with the (sometimes only conjectured) inequality 
\begin{align}
\hom(G,A)^{1/v(G)}\leq  \hom(K_{d+1},A)^{1/v(K_{d+1})}
\end{align}
for every $d$--regular graph $G$. It might occur very easily that there is no direct connection between these inequalities, but both of them are governed by certain correlation inequalities. In Section~\ref{Tutte-polynomial} we have seen that the FKG-inequality implies a positive correlation for the ferromagnetic Potts-model which in turn implies an inequality for $2$-lifts. This is also a case where it is conjectured that $K_{d+1}$ is the maximizing graph and it is proved for $d=3$, see \cite{DJPR2}. Another example for this phenomenon is the case of Widom--Rowlinson configurations where both the inequality 
$\hom(G\cup G,A_{\mathrm{WR}})\geq \hom(H,A_{\mathrm{WR}})$, and the extremality of $K_{d+1}$ hold true. 
Here we can also observe a certain  positive correlation. For the independent sets we have negative correlation and inequality of type $\hom(H,A)\leq \hom(G\times K_2,A)$, and the latter implies that the extremality of $K_{d,d}$ holds true. A possible intuition which may explain these phenomenons is the following: positive correlation implies that short cycles increases the number of homomorphisms, and negative correlation implies that short odd cycles decreases and short even cycles increases the number of homomorphisms. The most beautiful manifestation of this phenomenon is again the number of independent sets: if we want to minimize them then we have to have a lot of triangles in the graph, this  suggests $K_{d+1}$ (true! see \cite{CR}).  if we want to maximize them then we have to have a lot of $4$-cycles, but no triangles in the graph, this suggests $K_{d,d}$ (true! see \cite{Kahn,Zhao1}). For bipartite graphs we can only only prohibit short cycles, because it won't contain odd cycles, this suggests $\mathbb{T}_d$ (true!, this paper). If we want to minimize the number of independent sets in triangle-free graphs we have to find a graph without triangles and $4$-cycles, but with many $5$-cyles. For $d=3$ a natural candidate is the Petersen-graph. This is exactly the result of Perarnau and  Perkins \cite{PP}. Finally, if we want to maximize the number of independent sets with girth at least $5$ then we have to find a graph without $4$ and $5$-cycles, but with many $6$-cyles. For $d=3$ a natural candidate is the Heawood--graph. This is exactly another result of Perarnau and  Perkins \cite{PP}.
The author would not be surprised that if we want to minimize the number of independent sets with constraint girth at least $6$  then the Coxeter--graph would be the minimizer, and if we want to maximize the number of independent sets with constraint girth at least $7$ then the Tutte--Coxeter--graph would be the maximizer.
\bigskip

In this paper we were primarily interested in graph homomorphisms, but one can study other graph parameters with this method, for instance, the number of spanning trees or forests or other evaluations of the Tutte-polynomial. For the number of spanning trees $\tau(G)$, B. Mckay \cite{mckay} proved that for a $d$--regular graph $G$ on $n$ vertices we have
\begin{align}
\tau(G)\leq \frac{c \ln n}{n}\left( \frac{(d-1)^{d-1}}{(d^2-2d)^{d/2-1}}\right)^n.
\end{align}
With a slight modification of the method this problem can be handled with $2$-lifts, but we will get a weaker subexponential term. One might try to prove that $\tau(H)\geq \tau(G)^2$ for a $2$-lift $H$ of $G$. Unfortunately, this is not true and one should prove instead that $\mathbb{E}\tau(H)\geq \tau(G)^2$, where $\mathbb{E}\tau(H)$ is the average for all $2$-lifts of $G$. Fortunately, it is enough to deduce McKay's result. (But we have to admit that it will neither give a simpler proof, nor a better result.)  Note that R. Lyons \cite{lyons} proved the corresponding graph limit theorem. The effect of short cycles is very explicit in the work of B. McKay, and a certain negative correlation inequality is known as the consequence of Rayleigh's principle. For the number of forests $F(G)$, the author conjectures  that $F(H)\geq F(G)^2$ holds true for every graph $G$ and its $2$-lift $H$. This would follow from a well-known conjecture about a negative correlation inequality for the number of forests. The Reader might have already noticed that in these cases the inequalities are in the opposite directions, and $\mathbb{T}_d$ is the maximizing graph for the number of spanning trees. This is strongly related with certain phase transition for the Potts-model at $q=1$. 
\bigskip

{\footnotesize{
\begin{center}
\begin{tabular}{|l|cr|cr|cr|cr|} \hline
 & \multicolumn{4}{|c|}{All graphs} & \multicolumn{4}{|c|}{Bipartite graphs} \\ \hline
Graph parameter $P(G)$ & \multicolumn{2}{|c|}{Supremum} & \multicolumn{2}{|c|}{Infimum} & \multicolumn{2}{|c|}{Supremum} & \multicolumn{2}{|c|}{Infimum} \\ \hline
$\hom(G,H)$ (fixed $H$)& & & & & $K_{d,d}$ & \cite{GT}&  &\\ \hline
Number of independent sets  ($\mathcal{B}$)& $K_{d,d}$ & \cite{Zhao1}, $\mathcal{B}'$: \cite{Zhao2}& $K_{d+1}$ & \cite{CR} & $K_{d,d}$ & \cite{Kahn}& $\mathbb{T}^b_d$ & ($\mathcal{B}$) \\ \hline
Number of $q$-colorings & $K_{d,d}$ & (conj.) & $K_{d+1}$ & \cite{Zhao3}&  $K_{d,d}$ & \cite{GT}& $\mathbb{T}^b_d$ & \cite{CL} \\ \hline
Widom-Rowlinson (class $\mathcal{A,C}$)& $K_{d+1}$ & ($\mathcal{C}$: \cite{CPT,CCPT,S}) & $\mathbb{T}_d$ & ($\mathcal{A}$)& $K_{d,d}$ & \cite{GT}& $\mathbb{T}^b_d$ & ($\mathcal{A}$)\\ \hline
Number of perfect matchings & $K_{d,d}$ & \cite{B}& 0 & & $K_{d,d}$ & \cite{B}& $\mathbb{T}^b_d$ & \cite{Sc}\\ \hline
Number of all matchings & $K_{d,d}$ & \cite{DJPR}& & & $K_{d,d}$ & \cite{DJPR}& $\mathbb{T}^b_d$ & \cite{G,csi1}\\ \hline
Number of spanning trees & $\mathbb{T}_d$ & \cite{mckay}& & & $\mathbb{T}^b_d$ & \cite{mckay}& & \\ \hline
Number of forests & $\mathbb{T}_d$ & (conj.) & $K_{d+1}$ & (conj.) & $\mathbb{T}^b_d$ & (conj.) & &\\ \hline
Number of acyclic orientations & $\mathbb{T}_d$ & (conj.) & $K_{d+1}$ & \cite{GVS} & $\mathbb{T}^b_d$ &  (conj.) & &\\ \hline
\end{tabular}
\end{center}
\bigskip
}}

\subsection{Open problems.} There are two open problems which naturally arise in the study of $2$-lifts and large girth graphs.

\begin{?} \label{2-lift_question}
Is it true that for any bipartite graph $G$ and its $2$-lift $H$, and target graph (matrix) $A$ we have
\begin{align}
Z(G\cup G,A)\geq Z(H,A)?
\end{align}
In other words, is it true that every non-negative symmetric matrix is in $\mathcal{A}^b$?
\end{?}

\begin{?} \label{limit_question}
Is it true that if $(G_i)$ is a sequence of $d$--regular bipartite graphs such that $g(G_i)\to \infty$ then
\begin{align}
\lim_{i\to \infty}Z(G_i,A)^{1/v(G_i)}
\end{align}
exists?
\end{?}

Note that Problem~\ref{limit_question} has many natural variants. One can ask whether it is true that if a sequence of bipartite graphs $(G_i)$  converges to a random unimodular tree then 
\begin{align}
\lim_{i\to \infty}Z(G_i,A)^{1/v(G_i)}
\end{align}
exists or not. Or an even more optimistic question that if a sequence of bipartite graphs $(G_i)$  converges to a random unimodular graph (so not necessarily to a tree) then 
\begin{align}
\lim_{i\to \infty}Z(G_i,A)^{1/v(G_i)}
\end{align}
exists or not.
\bigskip

\noindent \textbf{Acknowledgment.} We are very grateful to the following people for discussions on topics related to this paper: Mikl\'os Ab\'ert, Emma Cohen, Jonathan Cutler, David Galvin, Yang Liu, Will Perkins, Nicholas Ruozzi, Luke Sernau, Prasad Tetali, and Yufei Zhao.
\bigskip

\bibliographystyle{siamnodash}

\bibliography{all_bibliography}

\end{document}